\documentclass[a4paper, 12pt, leqno]{amsart}
\usepackage{amsmath}
\usepackage{amscd}
\usepackage{amssymb}
\usepackage{amsthm}
\usepackage{amsfonts}
\usepackage{latexsym}
\newtheorem{lemma}{{\sc Lemma}}[section]
\newtheorem{corollary}[lemma]{{\sc Corollary}}
\newtheorem{proposition}[lemma]{{\sc Proposition}}
\newtheorem{theorem}[lemma]{{\sc Theorem}}
\newtheorem{remark}[lemma]{{\sc Remark}}

\numberwithin{equation}{section}

\def\Ga{{\mathfrak{a}}}
\def\Gb{{\mathfrak{b}}}

\def\Gf{{\mathfrak{f}}}
\def\Gg{{\mathfrak{g}}}
\def\Gh{{\mathfrak{h}}}
\def\Gk{{\mathfrak{k}}}
\def\Gl{{\mathfrak{l}}}
\def\Gm{{\mathfrak{m}}}
\def\Gn{{\mathfrak{n}}}

\def\Gs{{\mathfrak{s}}}

\def\BA{{\mathbb{A}}}
\def\BB{{\mathbb{B}}}
\def\BC{{\mathbb{C}}}
\def\BF{{\mathbb{F}}}

\def\BZ{{\mathbb{Z}}}

\def\CB{{\mathcal B}}

\def\CH{{\mathcal H}}
\def\CO{{\mathcal O}}
\def\CI{{\mathcal I}}
\def\CJ{{\mathcal J}}

\def\CR{{\mathcal R}}

\def\Ad{\mathop{\rm Ad}\nolimits}

\def\deru{\partial}

\def\End{\mathop{\rm{End}}\nolimits}

\def\Hom{\mathop{\rm Hom}\nolimits}

\def\Image{\mathop{\rm Im}\nolimits}
\def\Ker{\mathop{\rm Ker\hskip.5pt}\nolimits}

\begin{document}
\title[Manin triples and quantum groups]
{Manin triples and differential operators on quantum groups}
\author{Toshiyuki TANISAKI}
\dedicatory{To Jiro Sekiguchi on his 60th birthday}
\thanks
{
The author was partially supported by the Grants-in-Aid for Scientific Research, Challenging Exploratory Research No.\ 21654005
from Japan Society for the Promotion of Science.
}
\address{
Department of Mathematics, Osaka City University, 3-3-138, Sugimoto, Sumiyoshi-ku, Osaka, 558-8585 Japan}
\email{tanisaki@sci.osaka-cu.ac.jp}
\subjclass[2000]{20G05, 17B37, 53D17}
\begin{abstract}
Let $G$ be a simple algebraic group over $\BC$.
By taking the quasi-classical limit of the ring of differential operators on the corresponding quantized algebraic group at roots of 1 we obtain a Poisson manifold  $\Delta G\times K$, where
$\Delta G$ is the subgroup of $G\times G$ consisting of the diagonal elements, and $K$ is a certain subgroup of $G\times G$.
We show that this Poisson structure coincides with the one introduced by Semenov-Tyan-Shansky geometrically in the framework of Manin triples.
\end{abstract}
\maketitle

\section{Introduction}
In this paper we will explicitly compute the Poisson bracket of a certain Poisson manifold arising from the ring of differential operators on a quantized algebraic group at roots of 1.
This result will be a foundation in 
the author's recent works regarding the Beilinson-Bernstein type localization theorem for representations of quantized enveloping algebras at roots of 1 (see \cite{TDF}, \cite{TDF2}).

Let $G$ be a simple algebraic group over $\BC$ with Lie algebra $\Gg$.
Take Borel subgroups $B^+$ and $B^-$ of $G$ such that $H=B^+\cap B^-$ is a maximal torus of $G$.
Set $N^\pm=[B^\pm,B^\pm]$.
We define a subgroup $K$ of $G\times G$ by
\[
K=\{
(tx,t^{-1}y)\mid t\in H, x\in N^+, y\in N^-\}\subset B^+\times B^-\subset G\times G.
\]
Let $\zeta\in\BC^\times$ be a primitive $\ell$-th root of 1, where $\ell$ is an odd positive integer satisfying certain conditions depending on $\Gg$, and let $U_\zeta$ be the De Concini-Kac type quantized enveloping algebra  of $\Gg$ at $\zeta$.
It is expected that there exists a certain correspondence between representations of $U_\zeta$ and modules over the ring $D_{\CB_\zeta}$ of differential operators on the quantized flag manifold $\CB_\zeta$.
Since $D_{\CB_\zeta}$ is closely related to the ring $D_{G_\zeta}$ of  differential operators on
the quantized algebraic group $G_\zeta$, it is an important step in establishing the expected correspondence 
to investigate  the ring $D_{G_\zeta}$ in detail.
Note that $D_{G_\zeta}$ is nothing but the Heisenberg double $\BC[G_\zeta]\otimes U_\zeta$ of the Hopf algebras $\BC[G_\zeta]$ and $U_\zeta$, where $\BC[G_\zeta]$ is the coordinate algebra of $G_
\zeta$.
We have natural central embeddings
$
\BC[G]\subset\BC[G_\zeta]$, 
$
\BC[K]\subset U_\zeta
$
of Hopf algebras, 
and hence $G$ and $K$ become Poisson algebraic groups.
By De Concini-Procesi \cite{DP} and De Concini-Lyubashenko \cite{DL}
these Poisson algebraic group structures  of $G$ and $K$  turn out to be the ones defined geometrically from the Manin triple 
$(G\times G,\Delta G, K)$, where $\Delta G$ is the subgroup of $G\times G$ consisting of diagonal elements.
The aim of the present paper is to give a description of the Poisson algebra structure of $\BC[G]\otimes \BC[K]$ induced by the central embedding 
\begin{equation}
\label{eq:c-emb}
\BC[G]\otimes \BC[K]\subset
\BC[G_\zeta]\otimes U_\zeta
\end{equation}
of algebras.

Let
$(\Ga,\Gm,\Gl)$ be a Manin triple over $\BC$.
Assume that we are given a connected algebraic group
$A$ with Lie algebra $\Ga$ and connected closed subgroups $M$ and $L$ of $A$ with Lie algebras $\Gm$ and $\Gl$ respectively.
Then Semenov-Tyan-Shansky \cite{STS1}, \cite{STS2} 
showed that $A$ has a natural structure of Poisson manifold.
Hence by considering the pull-back with respect to the local isomorphism $M\times L\to A\,\,((m,l)\mapsto ml)$ the manifold $M\times L$ also turns out to be a Poisson manifold.
\begin{theorem}
The Poisson structure of $G\times K$ induced from the central embedding \eqref{eq:c-emb} coincides with the one defined geometrically from the Manin triple $(G\times G,\Delta G,K)$.
\end{theorem}
As explained above, the coincidence of the two Poisson brackets 
\[
\BC[G\times K]\times\BC[G\times K]\to \BC[G\times K]
\]
is already known for the parts $\BC[G]\times\BC[G]\to\BC[G]$ and $\BC[K]\times\BC[K]\to\BC[K]$ by \cite{DP}, \cite{DL}.
Hence we will be only concerned with the mixed part of the Poisson bracket between $\BC[G]$ and $\BC[K]$.
We point out that a closely related result in the case of  $\zeta=1$ for general Manin triples already appeared in \cite{STS2}.

In \cite{STS2} it is noted that the Poisson manifold $L$ associated to a Manin triple $(\Ga,\Gm,\Gl)$ can also be recovered as a Hamiltonian reduction with respect to the action of $M$ on $M\times L$.
In order to pass from $D_{G_\zeta}$ to $D_{\CB_\zeta}$ we need to consider Hamiltonian reduction for more general situation.
As a result we obtain the following.
\begin{proposition}
\label{prop:H-red}
The varieties
\begin{align*}
\overline{Y}=&\{(N^-g,(k_1,k_2))\in (N^-\backslash G)\times K
\mid
gk_1k_2^{-1}g^{-1}\in HN^-
\},\\
\overline{Y}_t=&\{(B^-g,(k_1,k_2)\in (B^-\backslash G)\times K
\mid
gk_1k_2^{-1}g^{-1}\in tN^-
\}\qquad(t\in H)
\end{align*}
turn out to be Poisson manifolds with respect to the Poisson tensors induced from that of $G\times K$.
Moreover, the Poisson tensors of  $\overline{Y}$ and $\overline{Y}_t$ are non-degenerate.
Hence they are symplectic manifolds.
\end{proposition}
In fact the Poisson manifold arising from
the Poisson structure of the center of $D_{\CB_\zeta}$ coincides with $\overline{Y}$ above (see \cite{TDF}).
The non-degeneracy of the Poisson tensor plays a crucial role in  the argument of \cite{TDF}.

The contents of this paper is as follows.
In Section 2 we recall the definition of the Poisson structure due to Semenov-Tyan-Shansky, and show that the technique of the Hamiltonian reduction works for certain cases.
The case of the typical Manin triple $(\Gg\oplus\Gg,\Delta\Gg,\Gk)$ is discussed in detail.
In Section 3 we give a summary of some of the known results on quantized enveloping algebras at roots of 1 due to Lusztig \cite{L2}, De Concini-Kac \cite{DK}, De Concini-Lyubashenko \cite{DL}, De Concini-Procesi \cite{DP}, Gavarini \cite{Gav}.
In Section 4 we show that the Poisson structure arising from the algebra of differential operators acting on quantized coordinate algebra of $G$ at roots of 1 coincides with the one coming from the typical Manin triple.

\section{Poisson structures arising from Manin triples}
\subsection{Manin triples}
We first recall standard facts on Poisson structures (see e.g.\ \cite{Drinfeld}, \cite{DP}).
A commutative associative algebra $\CR$ over $\BC$ equipped with a bilinear map
$\{\,,\,\}:\CR\times \CR\to \CR$
is called a Poisson algebra if it satisfies
\begin{itemize}
\item[(a)]
$\{a,a\}=0\quad(a\in \CR)$,
\item[(b)]
$\{a,\{b,c\}\}+\{b,\{c,a\}\}+\{c,\{a,b\}\}=0\quad(a, b, c\in \CR)$,
\item[(c)]
$\{a,bc\}=b\{a,c\}+\{a,b\}c\quad(a, b, c\in \CR)$.
\end{itemize}
A map $F:\CR\to \CR'$ between Poisson algebras $\CR$, $\CR'$ is called a homomorphism of Poisson algebras if it is a homomorphism of associative algebras and satisfies
$F(\{a_1,a_2\})=\{F(a_1),F(a_2)\}$ for any $a_1, a_2\in \CR$.
The tensor product $\CR\otimes_\BC \CR'$ of two Poisson algebras $\CR$, $\CR'$ over $\BC$ is equipped with a canonical Poisson algebra structure given by 
\begin{align*}
&(a_1\otimes b_1)(a_2\otimes b_2)=a_1a_2\otimes b_1b_2,\\
&\{a_1\otimes b_1,a_2\otimes b_2\}=\{a_1,a_2\}\otimes b_1b_2+a_1a_2\otimes \{b_1,b_2\}
\end{align*}
for $a_1, a_2\in \CR$, $b_2, b_2\in \CR'$.
A commutative Hopf algebra $\CR$  over a field $\BC$ equipped with a bilinear map
$
\{\,,\,\}:\CR\times \CR\to \CR
$
is called a Poisson Hopf algebra if it is a Poisson algebra and the comultiplication $\CR\to \CR\otimes_\BC \CR$ is a homomorphism of Poisson algebras (in this case the counit $\CR\to\BC$ and the antipode $\CR\to \CR$ become automatically a homomorphism and an anti-homomorphism of Poisson algebras respectively).

For a smooth algebraic variety $X$ over $\BC$ let
$\CO_X$ (resp.\ $\Theta_X$, $\Omega_X$) be the sheaf of regular functions (resp.\ vector fields, 1-forms). 
We denote the tangent and the cotangent bundles of $X$ by $TX$ and $T^*X$ respectively.
A smooth affine algebraic variety $X$ over $\BC$ is called a Poisson variety if we are given a bilinear map
$
\{\,,\,\}:\BC[X]\times\BC[X]\to\BC[X]
$
so that $\BC[X]$ is a Poisson algebra.
In this case $\{f, g\}(x)$ for $f, g\in\BC[X]$ and $x\in X$ depends only on $df_x, dg_x$, and hence we have $\delta\in\Gamma(X,\bigwedge^2\Theta_X)$ (called the Poisson tensor of the Poisson variety $X$)
such that 
\[
\{f, g\}(x)=\delta_x(df_x, dg_x).
\]
Consequently we also have the notion of Poisson variety which is not necessarily affine.

Let $S$ be a linear algebraic group over $\BC$ with Lie algebra $\Gs$.
For $a\in\Gs$ we define vector fields
$R_a, L_a\in\Gamma(S,\Theta_S)$ by
\begin{align*}
(R_a(f))(s)&=\frac{d}{dt}f(\exp(-ta)s)|_{t=0}
&(f\in\CO_S,\,s\in S),\\
(L_a(f))(s)&=\frac{d}{dt}f(s\exp(ta))|_{t=0}
&(f\in\CO_S,\,s\in S).
\end{align*}
For $\xi\in\Gs^*$ we also define 1-forms $L^*_\xi, R^*_\xi\in\Gamma(S,\Omega_S)$ by
\[
\langle L^*_\xi, L_a\rangle=
\langle R^*_\xi, R_a\rangle=
\langle \xi, a\rangle
\qquad(a\in\Gs).
\]
For $s\in S$ we define $\ell_s:S\to S$ by $\ell_s(x)=sx$.

A linear algebraic group $S$ over $\BC$ is called a Poisson algebraic group if we are given a bilinear map
$
\{\,,\,\}:\BC[S]\times\BC[S]\to\BC[S]
$
so that $\BC[S]$ is a Poisson Hopf algebra.
Let $\delta$ be the Poisson tensor of $S$ as a Poisson variety, and define $\varepsilon:S\to \bigwedge^2\Gs$ by 
$(d\ell_s)(\varepsilon(s))=\delta_s$ for $s\in S$.
Here, we identify the tangent space $(TS)_1$ at the identity element $1\in S$ with $\Gs$ by $L_a\leftrightarrow a\,\,(a\in\Gs)$.
By differentiating $\varepsilon$ at $1$ we obtain a linear map $\Gs\to\bigwedge^2\Gs$.
It induces an alternating bilinear map $[\,,\,]:\Gs^*\times\Gs^*\to\Gs^*$.
Then this $[\,,\,]$ gives a Lie algebra structure of $\Gs^*$.
Moreover, the following bracket product gives a Lie algebra structure of $\Gs\oplus\Gs^*$:
\[
[(a,\varphi),(b,\psi)]
=([a,b]+\varphi b-\psi a, a\psi-b\varphi+[\varphi,\psi]).
\]
Here, $\Gs\times \Gs^*\ni(a,\varphi)\to a\varphi\in\Gs^*$ and 
$\Gs^*\times \Gs\ni(\varphi, a)\to \varphi a\in\Gs$ are the coadjoint actions of $\Gs$ and $\Gs^*$ on $\Gs^*$ and $\Gs$ respectively.
In other words $(\Gs\oplus\Gs^*,\Gs,\Gs^*)$ is a Manin triple with respect to the symmetric bilinear form $\rho$ on $\Gs\oplus\Gs^*$ given by 
$\rho((a,\varphi),(b,\psi))=\varphi(b)+\psi(a)$.
We say that $(\Ga,\Gm,\Gl)$ is a Manin triple with respect to a symmetric bilinear form $\rho$ on $\Ga$ if 
\begin{itemize}
\item[(a)]
$\Ga$ is a finite-dimensional Lie algebra,
\item[(b)]
$\rho$ is $\Ga$-invariant and non-degenerate,
\item[(c)]
$\Gm$ and $\Gl$ are subalgebras of $\Ga$ such that $\Ga=\Gm\oplus\Gl$ as a vector space,
\item[(d)]
$\rho(\Gm,\Gm)=\rho(\Gl,\Gl)=\{0\}$.
\end{itemize}

Conversely, for each Manin triple we can associate a Poisson algebraic group 
 by reversing the above process as follows.
 Let $(\Ga,\Gm,\Gl)$ be a Manin triple with respect to a bilinear form $\rho$ on $\Ga$ and let $M$ be a linear algebraic group with Lie algebra $\Gm$.
Denote by
$
\pi_\Gm:\Ga\to\Gm,\,\,
\pi_\Gl:\Ga\to\Gl
$
the projections with respect to the direct sum decomposition $\Ga=\Gm\oplus\Gl$.
We sometimes identify $\Gm^*$ and $\Gl^*$ with $\Gl$ and $\Gm$ respectively via the non-degenerate bilinear form
$
\rho|_{\Gm\times\Gl}:\Gm\times\Gl\to\BC.
$
Hence we have also a natural identification
\begin{equation}
\label{eq:a-a-star}
\Ga^*=(\Gm\oplus\Gl)^*\cong\Gm^*\oplus\Gl^*
\cong\Gl\oplus\Gm=\Ga.
\end{equation}
For $m\in M$ we denote by $\Ad(m):\Ga\to\Ga$ the adjoint action.
Then we have the following (see e.g.\ \cite{Drinfeld}, \cite{DP}).
\begin{proposition}
\label{prop:deltaM}
The algebraic group $M$ is endowed with a structure of Poisson algebraic group whose Poisson tensor $\delta^M$ is given by
\begin{align*}
&
\delta^M_m(L^*_\xi,L^*_\eta)
=\rho(\pi_\Gm(\Ad(m)(\xi)),\Ad(m)(\eta))
&(\xi,\eta\in\Gl=\Gm^*),\\
&
\delta^M_m(R^*_\xi,R^*_\eta)
=-\rho(\pi_\Gm(\Ad(m^{-1})(\xi)),\Ad(m^{-1})(\eta))
&(\xi,\eta\in\Gl=\Gm^*)
\end{align*}
for $m\in M$.
\end{proposition}

\subsection{Semenov-Tyan-Shansky Poisson structure}
Let $(\Ga,\Gm,\Gl)$ be a Manin triple over $\BC$ with respect to a bilinear form $\rho$ on $\Ga$.
We assume that we are given a connected algebraic group $A$ and its closed connected subgroups $M$ and $L$ with Lie algebras $\Ga, \Gm, \Gl$ respectively.
Define an alternating bilinear form $\omega$ on $\Ga$ by
\[
\omega(a+b, a'+b')=\rho(a,b')-\rho(b,a')
\qquad(a, a'\in\Gm, b, b'\in\Gl).
\]
Denote the adjoint action of $A$ on $\Ga$ by
$\Ad:A\to GL(\Ga)$.
\begin{proposition}
[Semenov-Tyan-Shansky \cite{STS1}, \cite{STS2}]
\label{prop:STS}
The smooth affine variety $A$ is endowed with a structure of Poisson variety whose Poisson tensor $\tilde{\delta}$ is given by
\[
\tilde{\delta}_g(L^*_\xi,L^*_\eta)
=\frac12
(
\omega(\Ad(g)(\xi),\Ad(g)(\eta))+\omega(\xi,\eta)
)
\qquad
(\xi, \eta\in\Ga^*, g\in A).
\]
Here, we identify $\Ga$ with $\Ga^*$ via \eqref{eq:a-a-star}.
\end{proposition}
Note that we can rewrite $\tilde{\delta}$ in terms of $\rho$ as 
\begin{align*}
\tilde{\delta}_{g}(R^*_a,R^*_b)
=&\rho(a,(-\pi_\Gm+\Ad(g)\pi_\Gl\Ad(g^{-1}))(b))\\
=&\rho(a,(\pi_\Gl-\Ad(g)\pi_\Gm\Ad(g^{-1}))(b)),\\
\tilde{\delta}_{g}(L^*_a,L^*_b)=
&\rho(a,(-\pi_\Gm+\Ad(g^{-1})\pi_\Gl\Ad(g))(b))\\
=&\rho(a,(\pi_\Gl-\Ad(g^{-1})\pi_\Gm\Ad(g))(b))
\qquad(g\in A, a, b\in\Ga).
\end{align*}
Consider the map
\begin{equation}
\Phi:M\times L\to A\qquad((m,l)\mapsto ml).
\end{equation}
Since $\Phi$ is a local isomorphism, we obtain a Poisson structure of $M\times L$ whose Poisson tensor $\delta$ is the pull-back of $\tilde{\delta}$ with respect to $\Phi$.
Let us give a concrete description of $\delta$.
By Proposition \ref{prop:deltaM} $M$ is endowed with a structure of Poisson algebraic group.
By the symmetry of the notion of a Manin triple $L$ is also a Poisson algebraic group whose Poisson tensor $\delta^L$ is given by
\begin{align*}
&
\delta^L_l(L^*_\xi,L^*_\eta)
=\rho(\pi_\Gl(\Ad(l)(\xi)),\Ad(l)(\eta))
&(l\in L, \xi,\eta\in\Gm=\Gl^*),\\
&
\delta^L_l(R^*_\xi,R^*_\eta)
=-\rho(\pi_\Gl(\Ad(l^{-1})(\xi)),\Ad(l^{-1})(\eta))
&(l\in L, \xi,\eta\in\Gm=\Gl^*).
\end{align*}
By a standard computation we have the following.
\begin{proposition}
\label{prop:PoissonML}
The Poisson tensor $\delta$ is given by
\[
\delta_{(m,l)}:
((T^*M)_m\oplus(T^*L)_l)\times((T^*M)_m\oplus(T^*L)_l)\to\BC
\]
for $(m,l)\in M\times L$
with
\begin{align}
\label{eq:PML1}
&\delta_{(m,l)}|_{(T^*M)_m\times(T^*M)_m}=\delta^M_m,\\
\label{eq:PML2}
&\delta_{(m,l)}|_{(T^*L)_l\times(T^*L)_l}=\delta^L_l,\\
\label{eq:PML3}
&\delta_{(m,l)}(L^*_a,R^*_\xi)=\rho(a,\xi)
&
(a\in\Gl=\Gm^*, \xi\in\Gm=\Gl^*).
\end{align}
\end{proposition}

As noted in \cite{STS2} the Poisson tensors $\tilde{\delta}$ and $\delta$ are non-degenerate at generic points, and hence 
some open subsets of $A$ and $M\times L$ turn out to be symplectic manifolds.
We give below the condition on the point of $A$ and $M\times L$  so that the Poisson tensor is non-degenerate.

\begin{lemma}
\label{lem:non-deg}
\begin{itemize}
\item[{\rm(i)}]
Let $g\in A$.
Then 
$\tilde{\delta}_g$ is non-degenerate if and only if 
\[
\Ad(g)(\Gl)\cap\Gm=\Ad(g)(\Gm)\cap\Gl=\{0\}.
\]
\item[{\rm(ii)}]
Let $(m,l)\in M\times L$.
Then we have
\begin{align*}
\dim{\rm rad}\,\delta_{(m,l)}
=\dim(\Gl\cap\Ad(ml)(\Gm)).
\end{align*}
Especially, 
$\delta_{(m,l)}$ is non-degenerate if and only if
\[
\Ad(m^{-1})(\Gl)\cap\Ad(l)(\Gm)=\{0\}.
\]
\end{itemize}
\end{lemma}
\begin{proof}
(i) Set $F=-\pi_\Gm+\Ad(g)\pi_\Gl\Ad(g^{-1}):\Ga\to\Ga$ for simplicity.
By definition $\tilde{\delta}_g$ is non-degenerate if and only if $F$ is an isomorphism. 

Assume that $F$ is an isomorphism.
Since $F$ is surjective, we must have $\Ga=\Gm+\Ad(g)(\Gl)$ by the definition of $F$.
By $\dim\Ga=\dim\Gm+\dim\Gl$ we have $\Ga=\Gm\oplus\Ad(g)(\Gl)$
and $\Gm\cap\Ad(g)(\Gl)=0$.
Then
\[
\Ker F
=\{a\in\Ga\mid \pi_\Gm(a)=\Ad(g)\pi_\Gl\Ad(g^{-1})(a)=0\}
=\Gl\cap\Ad(g)(\Gm).
\]
Hence the injectivity of $F$ implies $\Gl\cap\Ad(g)(\Gm)=\{0\}$.

Assume $\Ad(g)(\Gl)\cap\Gm=\Ad(g)(\Gm)\cap\Gl=\{0\}$.
By $\Ad(g)(\Gl)\cap\Gm=\{0\}$ we have
$\Ga=\Gm\oplus\Ad(g)(\Gl)$.
Then $\Ker F=\Gl\cap\Ad(g)(\Gm)=\{0\}$.
Hence $F$ is an isomorphism.

(ii) For $g=ml$ we have
\[
\Ad(g)(\Gl)\cap\Gm=
\Ad(m)(\Ad(l)(\Gl)\cap\Ad(m^{-1})(\Gm))
=\Ad(m)(\Gl\cap\Gm)=\{0\}.
\]
Hence by the proof of (i) we obtain
\begin{align*}
&\dim{\rm rad}\,\delta_{(m,k)}
=\dim\Ker(-\pi_\Gm+\Ad(g)\pi_\Gl\Ad(g^{-1}))\\
=&\dim(\Gl\cap\Ad(g)(\Gm)).
\end{align*}
\end{proof}
\begin{corollary}
\begin{itemize}
\item[{\rm(i)}]
The Poisson structure of $A$ induces a symplectic structure of 
the open subset 
\[
\tilde{U}=\{
g\in A\mid
\Ad(g)(\Gl)\cap\Gm=\Ad(g)(\Gm)\cap\Gl=\{0\}
\}
\]
of $A$
\item[{\rm(ii)}]
The Poisson structure of $M\times L$ induces a symplectic structure of 
the open subset 
\[
U:=\{(m,l)\in M\times L\mid
\Ad(m^{-1})(\Gl)\cap\Ad(l)(\Gm)=\{0\}
\}
\]
of $M\times L$.
\end{itemize}
\end{corollary}

\subsection{A variant of Hamiltonian reduction}
Let $X$ be a Poisson variety with Poisson tensor $\delta$ and let $S$ be a connected linear algebraic group acting on the algebraic variety $X$ (we do not assume that $S$ preserves the Poisson structure of $X$).
Assume also that we are given an $S$-stable smooth subvariety $Y$ of $X$ on which $S$ acts locally freely.
Denote by $\Gs$ the Lie algebra of $S$.

For $y\in Y$ the linear map
\[
\Gs\ni a\mapsto\deru_a\in(TY)_y,\qquad
(\deru_af)(y)=\frac{d}{dt}f(\exp(-ta)y)|_{t=0}
\]
is injective by the assumption.
Hence we may regard
$\Gs\subset(TY)_y$ for $y\in Y$.
This gives an embedding
\[
Y\times\Gs\subset TY\subset(TX|_Y)
\]
of vector bundles on $Y$.
Correspondingly, we have
\[
T^*_YX\subset(Y\times\Gs)^\perp\subset(T^*X|_Y)
\]
where
\[
(Y\times\Gs)^\perp
=\{v\in(T^*X|_Y)\mid\langle v,Y\times\Gs\rangle=0\},
\]
and $T^*_YX$ denotes the conormal bundle.

By restricting
$\delta\in\Gamma(\wedge^2(TX))$ to $Y$ we obtain
$\delta|_Y\in\Gamma(\wedge^2(TX|_Y))$.
For $y\in Y$ restricting the anti-symmetric bilinear form $(\delta|_Y)_y$ on 
$(T^*X)_y$ to $((Y\times\Gs)^\perp)_y$ we obtain an anti-symmetric bilinear form $\hat{\delta}_y$ on $((Y\times\Gs)^\perp)_y$.
Then we have $\hat{\delta}\in\Gamma(\wedge^2((TX|Y)/(Y\times\Gs)))$.
Denote the action of 
$g\in S$ by $r_g:X\to X$.
Then for $y\in Y$ the isomorphism 
$(dr_g)_y:(TX)_y\to(TX)_{gy}$
induces
\[
(dr_g)_y:(TY)_y\to(TY)_{gy},\qquad
(dr_g)_y:\Gs\ni a\mapsto\Ad(g)(a)\in\Gs,
\]
where $\Gs$ is identified with subspaces of $(TY)_y$ and $(TY)_{gy}$.
In particular, $S$ naturally acts on $\Gamma(\wedge^2((TX|Y)/(Y\times\Gs)))$.

\begin{proposition}
\label{prop:H-red0}
Assume that $\hat{\delta}$ is $S$-invariant and $(T^*_YX)_y\subset{\rm rad}(\hat{\delta}_y)$ for any $y\in Y$.
Then the quotient space $S\backslash Y$ admits a natural structure of Poisson variety as follows.
Let 
$\varphi, \psi$ be functions on $S\backslash Y$, and let $\tilde{\varphi}, \tilde{\psi}$ be the corresponding $S$-invariant functions on $Y$.
Take extensions $\hat{\varphi}, \hat{\psi}$ of $\tilde{\varphi}, \tilde{\psi}$ to $X$ (not necessarily $S$-invariant).
Then $\{\hat{\varphi}, \hat{\psi}\}|_Y$ is $S$-invariant and does not depend on the choice of $\hat{\varphi}, \hat{\psi}$.
We define $\{\varphi,\psi\}$ to be the function corresponding to $\{\hat{\varphi}, \hat{\psi}\}|_Y$.

Moreover, if we have $(T^*_YX)_y={\rm rad}(\hat{\delta}_y)$ for any $y\in Y$, then the Poisson tensor of $S\backslash Y$ is non-degenerate.
Hence $S\backslash Y$ turns out to be a symplectic variety.
\end{proposition}
\begin{proof}
For $F\in\CO_X$, $\deru\in\Theta_X$, $y\in Y$ we have
$\langle(dF)_y,\deru\rangle=(\deru(F))(y)$, 
and hence
$F|_Y$ is $S$-invariant (resp.\ $F|_Y$ is a locally constant function) if and only if $dF|_Y\in(Y\times\Gs)^\perp$ (resp.\
$dF|_Y\in T^*_YX$).

Take $\varphi, \psi$ and $\tilde{\varphi}, \tilde{\psi}$, $\hat{\varphi}, \hat{\psi}$ as above.
We first show that 
$\{\hat{\varphi}, \hat{\psi}\}|_Y$ does not depend on the choice of 
 $\hat{\varphi}, \hat{\psi}$.
 For that it is sufficient to show that 
 $\{\hat{\varphi}, \hat{\psi}\}|_Y=0$ if $\tilde{\psi}=0$.
 By $d\hat{\varphi}|_Y\in(Y\times \Gs)^\perp$, $d\hat{\psi}|_Y\in T^*_YX$ we have
\[
\{\hat{\varphi},\hat{\psi}\}(y)=\delta_y((d\hat{\varphi})_y,(d\hat{\psi})_y)
=\hat{\delta}_y((d\hat{\varphi})_y,(d\hat{\psi})_y)=0
\]
by the assumption.

Let us show that $\{\hat{\varphi}, \hat{\psi}\}|_Y$ is $S$-invariant.
For $g\in S, y\in Y$ we have
\begin{align*}
&\{\hat{\varphi}, \hat{\psi}\}(gy)
=\hat{\delta}_{gy}((d\hat{\varphi})_{gy},(d\hat{\psi})_{gy})
=\hat{\delta}_{y}(d(\hat{\varphi}\circ r_g)_{y},d(\hat{\psi}\circ r_g)_{y})\\
=&\{\hat{\varphi}\circ r_g,\hat{\psi}\circ r_g\}(y)
\end{align*}
by the $S$-invariance of $\hat{\delta}$.
Since $\tilde{\varphi}, \tilde{\psi}$ are $S$-invariant, we have
$\hat{\varphi}\circ r_g|_Y=\tilde{\varphi}$ and $\hat{\psi}\circ r_g|_Y=\tilde{\psi}$.
Hence the independence of $\{\hat{\varphi}, \hat{\psi}\}|_Y$ on the choice of $\hat{\varphi}, \hat{\psi}$ implies
\[
\{\hat{\varphi}\circ r_g,\hat{\psi}\circ r_g\}(y)
=\{\hat{\varphi},\hat{\psi}\}(y)
\]
for $g\in S$ and $y\in Y$.

The remaining assertions are now clear.
\end{proof}

Now we apply the above general result to our Poisson varieties $M\times L$ and $A$.

Assume that we are given a connected closed subgroup $F$ of $M$.
Let $\Gf$ be the Lie algebra of $F$ and set
$
\Gf^\perp=\{a\in\Ga\mid\rho(\Gf,a)=0\}.
$
The action 
$F\times A\ni(x,g)\mapsto xg\in A$ 
of $F$ on $A$
induces an injection
\[
\Gf\ni a\mapsto R_a\in(TA)_g\qquad(g\in A).
\]
Define a subbundle $(A\times\Gf)^\perp$ of $T^*A$ by
\[
((A\times\Gf)^\perp)_g=\{R^*_c\mid c\in\Gf^\perp\}\subset(T^*A)_g,
\]
and set 
$
\hat{\delta}=\tilde{\delta}|_{(A\times\Gf)^\perp\times (A\times\Gf)^\perp}.
$
\begin{lemma}
\label{lem:F-inv}
If
$\Gf^\perp\cap\Gl$ is a Lie subalgebra of $\Gl$, then $\hat{\delta}$ is $F$-invariant.
\end{lemma}
\begin{proof}
By definition $\hat{\delta}_g$ for $g\in A$ is given by
\[
\hat{\delta}_g(R^*_c,R^*_{c'})=
\rho(c,(-\pi_\Gm+\Ad(g)\pi_\Gl\Ad(g^{-1}))(c'))
\qquad(c, c'\in\Gf^\perp).
\]
On the other hand for $x\in F, g\in A$ the isomorphism $(T^*A)_g\cong(T^*A)_{xg}$ induced by the action of $x$ is given by 
\[
(T^*A)_g\cong(T^*A)_{xg}\qquad
(R^*_b\mapsto R^*_{\Ad(x)(b)}).
\]
Hence it is sufficient to show
\[
\rho(\Ad(x)(c),\pi_\Gm\Ad(x)(c'))
=\rho(c,\pi_\Gm(c'))
\qquad(x\in F,\,c, c'\in\Gf^\perp).
\]
Since $F$ is connected, this is equivalent to its infinitesimal counterpart
\[
\rho([a,c],\pi_\Gm(c'))+
\rho(c,\pi_\Gm([a,c']))
=0
\qquad(a\in \Gf,\,c, c'\in\Gf^\perp).
\]
Note that $\Gf^\perp=\Gm\oplus(\Gf^\perp\cap\Gl)$.
If $c\in\Gm$, then we have $[a,c]\in\Gm$ and hence
$\rho([a,c],\pi_\Gm(c'))=
\rho(c,\pi_\Gm([a,c']))=0$.
If $c'\in\Gm$, then
\[
\rho([a,c],\pi_\Gm(c'))+
\rho(c,\pi_\Gm([a,c']))
=
\rho([a,c],c')+
\rho(c,[a,c'])=0
\]
by the invariance of $\rho$.
Hence we may assume that $c, c'\in\Gf^\perp\cap\Gl$.
In this case we have
\begin{align*}
&\rho([a,c],\pi_\Gm(c'))+
\rho(c,\pi_\Gm([a,c']))=
\rho(c,\pi_\Gm([a,c']))
=
\rho(c,[a,c'])\\
=&-\rho([c',c],a)
\in\rho(\Gf^\perp\cap\Gl,\Gf)=0.
\end{align*}
\end{proof}
By Proposition \ref{prop:H-red0} and Lemma \ref{lem:F-inv} 
we have the following.
\begin{proposition}
\label{prop:Hamil-red}
Assume that $\Gf^\perp\cap\Gl$ is a Lie subalgebra of $\Gl$.
Let $V$ be an $F$-stable smooth subvariety of $A$ such that the action of $F$ on $V$ is locally free.
Assume also that for 
$g\in V$ we have
\[
{\rm{rad}}(\hat{\delta}_g)\supset
(T^*_VA)_g.
\]
Then 
$F\backslash V$ has a structure of Poisson variety whose Poisson bracket is defined as follows:
Let $\varphi,\psi$ be functions on $F\backslash V$, and denote by 
$\tilde{\varphi}, \tilde{\psi}$ the corresponding $F$-stable functions on $V$.
Take extensions
$\hat{\varphi}, \hat{\psi}$ of $\tilde{\varphi}, \tilde{\psi}$ respectively  to $A$.
Then 
$\{\hat{\varphi},\hat{\psi}\}|_{{V}}$ is $F$-stable and dose not depend on the choice of $\hat{\varphi}, \hat{\psi}$.
We define 
$\{\varphi,\psi\}$ to be the function on $F\backslash V$ corresponding to $\{\hat{\varphi},\hat{\psi}\}|_{{V}}$.

If, moreover,
\[
{\rm{rad}}(\hat{\delta}_g)=
(T^*_VA)_g
\]
holds for any $g\in V$, then the Poisson tensor of $F\backslash V$ is non-degenerate $($hence $F\backslash V$ turns out to be a symplectic variety$)$.
\end{proposition}

\subsection{A special case}
\label{subsec:special}
Let $G$ be a connected simple algebraic group over $\BC$, and let $H$ be its maximal torus.
We take Borel subgroups $B^+, B^-$ of $G$ such that $H=B^+\cap B^-$, and set $N^{\pm}=[B^{\pm},B^{\pm}]$.
Denote the Lie algebras of $G$, $H$, $B^\pm$, $N^\pm$ by
$\Gg$, $\Gh$, $\Gb^\pm$, $\Gn^\pm$.
Define subalgebras $\Delta\Gg$ and $\Gk$ of $\Gg\oplus\Gg$ by
\begin{align*}
\Delta\Gg&=\{(a,a)\mid a\in\Gg\},\\
\Gk&=\{(h+x,-h+y)\mid h\in\Gh, x\in\Gn^+, y\in\Gn^-\},
\end{align*}
and denote by $\Delta G$, $K$ the connected closed subgroups of $G\times G$ with Lie algebras $\Delta\Gg, \Gk$ respectively.
In particular,
$\Delta G=\{(g,g)\mid g\in G\}.
$
We fix an invariant non-degenerate symmetric bilinear form
$
\kappa:\Gg\times\Gg\to\BC,
$
and define a bilinear form 
$\rho:(\Gg\oplus\Gg)\times(\Gg\oplus\Gg)\to\BC
$
by
\[
\rho((a,b),(a',b'))=\kappa(a,a')-\kappa(b,b').
\]
Then $(\Gg\oplus\Gg,\Delta\Gg,\Gk)$ is a Manin triple with respect to the bilinear form $\rho$.

By Proposition \ref{prop:STS} (resp. Proposition \ref{prop:PoissonML})  we have a Poisson structure of $G\times G$ (resp.\ $\Delta G\times K$) with Poisson tensor $\tilde{\delta}$  (resp.\ ${\delta}$).
Moreover, the Poisson structure of $\Delta G\times K$ is the pull-back of that of $G\times G$ with respect to 
\[
\Phi:\Delta G\times K\to G\times G
\qquad(((g,g),(k_1,k_2))\mapsto(gk_1,gk_2)).
\]

\begin{lemma}
\[
\Image\Phi=\{(g_1,g_2)\in G\times G\mid
g_1^{-1}g_2\in N^+HN^-\}.
\]
\end{lemma}
\begin{proof}
We have
\[
(gk_1)^{-1}(gk_2)
=k_1^{-1}k_2\in N^+HN^-.
\]
Assume $g_1^{-1}g_2\in N^+HN^-$.
Then for $(k_1,k_2)\in K$ with
$k_1^{-1}k_2=g_1^{-1}g_2$ we have
\[
(g_1,g_2)=(g_1k_1^{-1},g_2k_2^{-1})(k_1,k_2)\in\Image\Phi.
\]
\end{proof}
\begin{proposition}
$\delta_{((g,g),(k_1,k_2))}$ is non-degenerate if and only if we have
$gk_1k_2^{-1}g^{-1}\in N^+HN^-$.
\end{proposition}
\begin{proof}
Note that 
\begin{equation}
\label{eq:s1}
\dim{\rm rad}(\delta_{((g,g),(k_1,k_2))})
=\dim(\Gk\cap\Ad(gk_1,gk_2)(\Delta\Gg))
\end{equation}
by Lemma \ref{lem:non-deg}.
In general for $(g_1,g_2)\in G\times G$ set
$
d(g_1,g_2):=\dim(\Gk\cap\Ad(g_1,g_2)(\Delta\Gg)).
$
For
$(k_1,k_2)\in K$ and $(g,g)\in\Delta G$ we have
\[
d((k_1,k_2)(g_1,g_2)(g,g))=d(g_1,g_2),
\]
and hence $d(g_1,g_2)$ is regarded as a function on $K\backslash(G\times G)/\Delta G$.
Denote by $W=N_G(H)/H$ the Weyl group of $G$.
A standard fact on simple algebraic groups tells us that for any $(g_1,g_2)\in G\times G$ there exists some $w\in W$ and $t\in H$ such that 
$K(g_1,g_2)\Delta G\ni (t\dot{w},1)$, where $\dot{w}$ is a representative of $w$.
By
\[
d(t\dot{w},1)=\dim(\Gk\cap\Ad(t\dot{w},1)(\Delta\Gg))
=\dim(\Ad((t\dot{w},1)^{-1})(\Gk)\cap\Delta\Gg),
\]
\[
\Ad((t\dot{w},1)^{-1})(\Gk)
=
\{(w^{-1}h+\dot{w}^{-1}x,-h+y)
\mid
h\in\Gh, x\in\Gn^+, y\in\Gn^-\}
\]
we see easily that $d(t\dot{w},1)=0$ if and only if $w=1$.
The assertion follows from this easily.
\end{proof}
\begin{corollary}
The Poisson structure of $\Delta G\times K$ induces a symplectic structure of the open subset
\[
U:=\{((g,g),(k_1,k_2))\in \Delta G\times K\mid
gk_1k_2^{-1}g^{-1}\in N^+HN^-
\}.
\]
\end{corollary}

Set
\begin{align*}
Y&=\{((g,g),(k_1,k_2))\in \Delta G\times K\mid
gk_1k_2^{-1}g^{-1}\in B^-
\}\subset U\subset\Delta G\times K,\\
\tilde{Y}&=\Phi(Y)\subset G\times G.
\end{align*}
Then we have
\begin{equation}
\tilde{Y}
=\{(g_1,g_2)\in G\times G\mid
g_1g_2^{-1}\in B^-, \,g_1^{-1}g_2\in N^+HN^-
\}.
\end{equation}
Moreover, setting
\[
\tilde{Z}=
\{(g,b)\in G\times B^-
\mid
g^{-1}b^{-1}g\in N^+HN^-
\}
\]
we have
\begin{equation}
\label{eq:z1}
\tilde{Y}\cong\tilde{Z}
\qquad
((g_1,g_2)\leftrightarrow(g_1,g_1g_2^{-1}),\quad
(g,b^{-1}g)\leftrightarrow(g,b)).
\end{equation}
Since $N^+HN^-$ is an open subset of $G$, $\tilde{Z}$ is open in $G\times B^-$.
In particular, $\tilde{Z}$ is a smooth variety.
Hence $\tilde{Y}$ is also smooth.
Define an action of
$N^-$ on $G\times G$ by
\begin{align*}
x(g_1,g_2)=(xg_1,xg_2)
\qquad(x\in N^-, (g_1,g_2)\in G\times G).
\end{align*}
Then $\tilde{Y}$ is $N^-$-invariant.
Moreover, \eqref{eq:z1} preserves the action of $N^-$, where the action of $N^-$ on $\tilde{Z}$ is given by
\begin{align*}
x(g,b)=(xg,xbx^{-1})\qquad
(x\in N^-, (g,b)\in\tilde{Z}).
\end{align*}
For 
$C\subset G$ such that 
$C\ni c\mapsto N^-c\in{N^-}\backslash G$ is an open embedding
we have
\begin{align*}
&\{(g,b)\in\tilde{Z}\mid g\in N^-C\}\\
=&\{(yc,yby^{-1})\mid y\in N^-, c\in C, b\in B^-, 
c^{-1}b^{-1}c\in N^+HN^-\}\\
\cong& N^-\times
\{(c,b)\in C\times B^-\mid 
c^{-1}b^{-1}c\in N^+HN^-\},
\end{align*}
and hence the action of $N^-$ on $\tilde{Z}$ is locally free.
Hence we have the following.
\begin{lemma}
\label{lem:z2}
$\tilde{Y}$ is a smooth variety, and 
the action of 
$N^-$ on $\tilde{Y}$ is locally free.
\end{lemma}
Set $\Delta\Gn^-=\{(a,a)\mid a\in\Gn^-\}$.
We have obviously the following.
\begin{lemma}
We have
\[
(\Delta\Gn^-)^\perp\cap\Gk=
\{(h,-h+y)\mid y\in\Gn^-\}.
\]
In particular, $(\Delta\Gn^-)^\perp\cap\Gk$ is a Lie subalgebra of $\Gk$.
\end{lemma}
For $(g_1,g_2)\in\tilde{Y}$ we have
\begin{align*}
T(G\times G)_{(g_1,g_2)}&=\{R_{(a_1,a_2)}\mid(a_1,a_2)\in\Gg\oplus\Gg\},\\
T^*(G\times G)_{(g_1,g_2)}&=\{R^*_{(u_1,u_2)}\mid(u_1,u_2)\in\Gg\oplus\Gg\},\\
\langle
R_{(a_1,a_2)},
R^*_{(u_1,u_2)}
\rangle&=\kappa(a_1,u_1)-\kappa(a_2,u_2).
\end{align*}
By \eqref{eq:z1} we have also
\[
(T\tilde{Y})
_{(g_1,g_2)}
=
\{
R_{(a,\Ad(g_2g_1^{-1})(a))}\mid a\in\Gg\}
\oplus
\{R_{(0,b)}\mid b\in\Gb^-\}
\]
for $(g_1,g_2)\in\tilde{Y}$.
By Lemma \ref{lem:z2} the natural map 
$\Gn^-\to(T\tilde{Y})_{(g_1,g_2)}$ is injective and is given by
\[
\Gn^-\ni c\mapsto R_{(c,c)}\in(T\tilde{Y})_{(g_1,g_2)}.
\]
Hence under the identification 
$\Gn^-\subset(T\tilde{Y})_{(g_1,g_2)}\subset T(G\times G)_{(g_1,g_2)}$ we have
\begin{align*}
(\Gn^-)^\perp
=&\{R^*_{(u_1,u_2)}\mid u_1-u_2\in\Gb^-\}
=\{R^*_{(u,u+v)}\mid u\in\Gg, v\in\Gb^-\},\\
((T\tilde{Y})
_{(g_1,g_2)})^\perp
=&
\{R^*_{(\Ad(g_2g_1^{-1})(y),y)}\mid y\in\Gn^-\}.
\end{align*}
\begin{lemma}
For $(g_1,g_2)\in\tilde{Y}$ we have
\[
{\rm rad}
\left(\tilde{\delta}_{(g_1,g_2)}|_{
(\Gn^-)^\perp\times
(\Gn^-)^\perp
}
\right)
=((T\tilde{Y})
_{(g_1,g_2)})^\perp.
\]
\end{lemma}
\begin{proof}
For $u\in\Gg, v\in\Gb^-$ we have 
$R^*_{(u,u+v)}
\in{\rm{rad}}
\left(\tilde{\delta}_{(g_1,g_2)}|_{
(\Gn^-)^\perp\times
(\Gn^-)^\perp}
\right)$ 
if and only if 
$\tilde{\delta}_{(g_1,g_2)}(R^*_{(a,a+b)},R^*_{(u,u+v)})=0$
for any $a\in\Gg, b\in\Gb^-$.
Setting
\[
(-\pi_{\Delta\Gg}+\Ad(g_1,g_2)\pi_\Gk\Ad(g_1^{-1},g_2^{-1}))(u,u+v)
=(x,y)
\]
we have
\[
\tilde{\delta}_{(g_1,g_2)}(R^*_{(a,a+b)},R^*_{(u,u+v)})=
\kappa(a,x)-\kappa(a+b,y)
=\kappa(a,x-y)-\kappa(b,y).
\]
Hence 
$R^*_{(u,u+v)}
\in{\rm{rad}}
\left(\tilde{\delta}_{(g_1,g_2)}|_{
(\Gn^-)^\perp\times
(\Gn^-)^\perp}
\right)$ if and only if
$x=y\in\Gn^-$.
By $(g_1,g_2)\in\Phi(\Delta G\times K)$ we have
$\Gg\oplus\Gg=\Delta\Gg\oplus\Ad(g_1,g_2)(\Gk)$.
Therefore,
\begin{align*}
&R^*_{(u,u+v)}
\in{\rm{rad}}
\left(\tilde{\delta}_{(g_1,g_2)}|_{
(\Gn^-)^\perp\times
(\Gn^-)^\perp}
\right)\\
\Longleftrightarrow\,
&\pi_{\Delta\Gg}(u,u+v)=(y,y)\,\,(\exists y\in\Gn^-),\,\,
\pi_\Gk\Ad(g_1^{-1},g_2^{-1})(u,u+v)=0\\
\Longleftrightarrow\,
&u\in\Gn^-,\,\,
\Ad(g_1^{-1},g_2^{-1})(u,u+v)\in\Delta\Gg\\
\Longleftrightarrow\,
&u\in\Gn^-,\,\,
v=\Ad(g_2g_1^{-1})(u)-u.
\end{align*}
It follows that
\[
{\rm{rad}}
\left(\tilde{\delta}_{(g_1,g_2)}|_{
(\Gn^-)^\perp\times
(\Gn^-)^\perp}
\right)
=
\{R^*_{(u,\Ad(g_2g_1^{-1})(u))}\mid u\in\Gn^-\}
=((T\tilde{Y})
_{(g_1,g_2)})^\perp.
\]
\end{proof}
By Proposition \ref{prop:Hamil-red} and the above argument
we obtain the following.
\begin{proposition}
We have a natural Poisson structure of 
$N^-\backslash\tilde{Y}$
whose Poisson tensor is non-degenerate and defined as follows 
$($hence $N^-\backslash\tilde{Y}$ turns out to be a symplectic variety$)$:
Let 
$\varphi,\psi$ be functions on $N^-\backslash\tilde{Y}$, and let $\tilde{\varphi}, \tilde{\psi}$ be the corresponding $N^-$-invariant functions on $\tilde{Y}$.
Take extensions $\hat{\varphi}, \hat{\psi}$ of $\tilde{\varphi}, \tilde{\psi}$ to $G\times G$.
Then $\{\hat{\varphi},\hat{\psi}\}|_{\tilde{Y}}$ is $N^-$-invariant and does not depend on the choice of $\hat{\varphi}, \hat{\psi}$.
We define $\{\varphi,\psi\}$ to be the function on $N^-\backslash\tilde{Y}$ corresponding to $\{\hat{\varphi},\hat{\psi}\}|_{\tilde{Y}}$.
\end{proposition}
By considering the pull-back to $Y$ via $\Phi$ we also obtain the following.
\begin{proposition}
Consider the action of $N^-$ on $Y$ given by
\[
x((g,g),(k_1,k_2))=
((xg,xg),(k_1,k_2))\qquad
(x\in N^-, ((g,g),(k_1,k_2))\in Y).
\]
Then 
we have a natural Poisson structure of 
$N^-\backslash{Y}$
whose Poisson tensor is non-degenerate and defined as follows 
$($hence $N^-\backslash{Y}$ turns out to be a symplectic variety$)$:
Let 
$\varphi,\psi$ be functions on $N^-\backslash{Y}$, and let $\tilde{\varphi}, \tilde{\psi}$ be the corresponding $N^-$-invariant functions on ${Y}$.
Take extensions $\hat{\varphi}, \hat{\psi}$ of $\tilde{\varphi}, \tilde{\psi}$ to $\Delta G\times K$.
Then $\{\hat{\varphi},\hat{\psi}\}|_{{Y}}$ is $N^-$-invariant and does not depend on the choice of $\hat{\varphi}, \hat{\psi}$.
We define $\{\varphi,\psi\}$ to be the function on $N^-\backslash{Y}$ corresponding to $\{\hat{\varphi},\hat{\psi}\}|_{{Y}}$.
\end{proposition}
Note that 
\begin{equation}
N^-\backslash Y
\cong
\{(N^-g,(k_1,k_2))\in (N^-\backslash G)\times K\mid
gk_1k_2^{-1}g^{-1}\in B^-\}.
\end{equation}

Fix $t\in H$ and set
\begin{align*}
Y_t&=\{((g,g),(k_1,k_2))\in \Delta G\times K\mid
gk_1k_2^{-1}g^{-1}\in tN^-
\}\subset U\subset\Delta G\times K.
\end{align*}
Then by a similar argument we have the following.
\begin{proposition}
Consider the action of $B^-$ on $Y_t$ given by
\[
x((g,g),(k_1,k_2))=
((xg,xg),(k_1,k_2))\,\,
(x\in B^-, ((g,g),(k_1,k_2))\in Y_t).
\]
Then 
we have a natural Poisson structure of 
$B^-\backslash{Y}_t$
whose Poisson tensor is non-degenerate and defined as follows 
$($hence $B^-\backslash{Y}_t$ turns out to be a symplectic variety$)$:
Let 
$\varphi,\psi$ be functions on $B^-\backslash{Y}_t$, and let $\tilde{\varphi}, \tilde{\psi}$ be the corresponding $B^-$-invariant functions on ${Y}_t$.
Take extensions $\hat{\varphi}, \hat{\psi}$ of $\tilde{\varphi}, \tilde{\psi}$ to $\Delta G\times K$.
Then $\{\hat{\varphi},\hat{\psi}\}|_{{Y}_t}$ is $B^-$-invariant and does not depend on the choice of $\hat{\varphi}, \hat{\psi}$.
We define $\{\varphi,\psi\}$ to be the function on $B^-\backslash{Y}_t$ corresponding to $\{\hat{\varphi},\hat{\psi}\}|_{{Y}_t}$.
\end{proposition}
Note that we have
\begin{equation}
B^-\backslash Y_t
\cong
\{(B^-g,(k_1,k_2))\in (B^-\backslash G)\times K\mid
gk_1k_2^{-1}g^{-1}\in tN^-\}.
\end{equation}

\section{Quantized enveloping algebras}

\subsection{Lie algebras}
In the rest of this paper we will use the notation of Section \ref{subsec:special}.
In particular, $\Gg$ is a finite-dimensional simple Lie algebra over $\BC$, and $G$ is a connected algebraic group with Lie algebra $\Gg$.
We further assume that $G$ is simply-connected and the symmetric bilinear form 
\begin{equation}
\label{eq:SBMh*}
(\,,\,):\Gh^*\times\Gh^*\to\BC
\end{equation}
induced by $\kappa$ satisfies 
$(\beta,\beta)/2=1$ for short roots $\beta$.
We denote by $\Delta\subset\Gh^*$, $Q\subset\Gh^*$, $\Lambda\subset\Gh^*$ and $W\subset GL(\Gh^*)$ the set of roots, the root lattice $\sum_{\alpha\in\Delta}\BZ\alpha$, the weight lattice and the Weyl group respectively.
By our normalization of \eqref{eq:SBMh*} we have
\[
(\Lambda,Q)\subset\BZ,\qquad
(\Lambda,\Lambda)\subset\frac1{|\Lambda/Q|}\BZ.
\]
For $\beta\in\Delta$ we set
\[
\Gg_{\beta}
=\{x\in\Gg\mid [h,x]=\beta(h)x\,\,(h\in \Gh)\}.
\]
We choose a system of positive roots $\Delta^+\subset\Gh^*$ so that $
\Gn^\pm
=\bigoplus_{\beta\in\Delta^+}\Gg_{\pm\beta}$.
Let $\{\alpha_i\}_{i\in I}$, $\{s_i\}_{i\in I}\subset W$ be the corresponding  sets of simple roots and  simple reflections  respectively.
Set 
\[
Q^+
=\sum_{\alpha\in\Delta^+}\BZ_{\geqq0}\alpha
=\bigoplus_{i\in I}\BZ_{\geqq0}\alpha_i
\subset \Gh^*.
\]
We denote the longest element of $W$ by $w_0$.
For each $i\in I$ we take ${e}_i\in\Gg_{\alpha_i}, {f}_i\in\Gg_{-\alpha_i}, h_i\in\Gh$ such that $[{e}_i,{f}_i]=h_i$ and $\alpha_i(h_i)=2$.

Define subalgebras $\Gk^0, \Gk^+, \Gk^-$ of $\Gk$ by
\[
\Gk^0=\{(h,-h)\mid h\in\Gh\},\quad
\Gk^+=\{(x,0)\mid x\in\Gn^+\},\quad
\Gk^-=\{(0,y)\mid y\in \Gn^-\}.
\]
Then we have $\Gk=\Gk^+\oplus\Gk^0\oplus\Gk^-$.
For $i\in I$ set
\[
x_i=(e_i,0)\in\Gk^+,\quad y_i=(0, f_i)\in\Gk^-,\quad
t_i=(
h_i,-h_i)\in\Gk^0.
\]
We denote by $K^0$, $K^\pm$ the connected closed subgroups of $K$ with Lie algebras $\Gk^0$, $\Gk^\pm$ respectively.

\subsection{Quantized enveloping algebra of $\Gg$}
For $n\in\BZ$ and $m\in\BZ_{\geqq0}$ we set
\begin{align*}
[n]_t&=\frac{t^n-t^{-n}}{t-t^{-1}}\in\BZ[t,t^{-1}],
\qquad
[m]_t!=[m]_t[m-1]_t\cdots[2]_t[1]_t\in\BZ[t,t^{-1}],\\
\begin{bmatrix}n\\m\end{bmatrix}_t&=
[n]_t[n-1]_t\cdots[n-m+1]_t/[m]_t!\in\BZ[t,t^{-1}].
\end{align*}

The quantized enveloping algebra $U=U_q(\Gg)$ of $\Gg$ is an associative algebra over  $\BF=\BC(q^{1/|\Lambda/Q|})$ with identity element $1$ generated by the elements $K_\lambda\,(\lambda\in \Lambda),\,E_i, F_i\,(i\in I)$ satisfying the following defining relations:
\begin{align}
&K_0=1,\quad 
K_\lambda K_\mu=K_{\lambda+\mu}
&(\lambda,\mu\in \Lambda),
\label{eq:def1}\\
&K_\lambda E_iK_\lambda^{-1}=q^{(\lambda,\alpha_i)}E_i&(\lambda\in \Lambda, i\in I),
\label{eq:def2a}\\
&K_\lambda F_iK_\lambda^{-1}=q^{-(\lambda,\alpha_i)}F_i
&(\lambda\in \Lambda, i\in I),
\label{eq:def2b}\\
&E_iF_j-F_jE_i=\delta_{ij}\frac{K_i-K_i^{-1}}{q_i-q_i^{-1}}
&(i, j\in I),
\label{eq:def3}\\
&\sum_{n=0}^{1-a_{ij}}(-1)^nE_i^{(1-a_{ij}-n)}E_jE_i^{(n)}=0
&(i,j\in I,\,i\ne j),
\label{eq:def4}\\
&\sum_{n=0}^{1-a_{ij}}(-1)^nF_i^{(1-a_{ij}-n)}F_jF_i^{(n)}=0
&(i,j\in I,\,i\ne j),
\label{eq:def5}
\end{align}
where $q_i=q^{(\alpha_i,\alpha_i)/2}, K_i=K_{\alpha_i}, a_{ij}=2(\alpha_i,\alpha_j)/(\alpha_i,\alpha_i)$ for $i, j\in I$, and
\[
E_i^{(n)}=
E_i^n/[n]_{q_i}!,
\qquad
F_i^{(n)}=
F_i^n/[n]_{q_i}!
\]
for $i\in I$ and $n\in\BZ_{\geqq0}$.
Algebra homomorphisms $\Delta:U\to U\otimes U, \varepsilon:U\to\BF$ and an algebra anti-automorphism $S:U\to U$ are defined by:
\begin{align}
&\Delta(K_\lambda)=K_\lambda\otimes K_\lambda,\\
&\Delta(E_i)=E_i\otimes 1+K_i\otimes E_i,\quad
\Delta(F_i)=F_i\otimes K_i^{-1}+1\otimes F_i,
\nonumber\\
&\varepsilon(K_\lambda)=1,\quad
\varepsilon(E_i)=\varepsilon(F_i)=0,\\
&S(K_\lambda)=K_\lambda^{-1},\quad
S(E_i)=-K_i^{-1}E_i, \quad S(F_i)=-F_iK_i,
\end{align}
and $U$ is endowed with a Hopf algebra structure with the comultiplication $\Delta$, the counit $\varepsilon$ and the antipode $S$.

We define subalgebras $U^0, U^{\geqq0}, U^{\leqq0}, U^{+}, U^{-}$ of $U$ by 
\begin{align}
U^0&=\langle K_\lambda\mid \lambda\in \Lambda\rangle,\\
U^{\geqq0}&=\langle K_\lambda, E_i\mid \lambda\in \Lambda, i\in I\rangle,\\
U^{\leqq0}&=\langle K_\lambda, F_i\mid \lambda\in \Lambda, i\in I\rangle,\\
U^{+}&=\langle E_i\mid i\in I\rangle,\\
U^{-}&=\langle F_i\mid i\in I\rangle.
\end{align}
The following result is standard.
\begin{proposition}
\label{prop:Str-of-U}
\begin{itemize}
\item[\rm(i)]
$\{K_\lambda\mid \lambda\in \Lambda\}$ is an $\BF$-basis of $U^0$.
\item[\rm(ii)]
The linear maps 
\begin{gather*}
U^-\otimes U^0\otimes U^+\to U \gets U^+\otimes U^0\otimes U^-,\\
U^+\otimes U^0\to U^{\geqq0} \gets U^0\otimes U^+,\qquad
U^-\otimes U^0\to U^{\leqq0} \gets U^0\otimes U^-
\end{gather*}
induced by the multiplication are all isomorphisms of vector spaces.
\end{itemize}
\end{proposition}

For $\gamma\in Q$ we set
\[
U^\pm_{\gamma}
=\{x\in U^\pm\mid
K_\lambda xK_\lambda^{-1}=q^{(\lambda,\gamma)}x\,\,(\lambda\in \Lambda)\}.
\]
We have $U^\pm_{\pm\gamma}=\{0\}$ unless $\gamma\in Q^+$, and 
\[
U^\pm=\bigoplus_{\gamma\in Q^+}U^\pm_{\pm\gamma},
\qquad
\dim U^\pm_{\pm\gamma}<\infty\quad(\gamma\in Q^+)
.
\]

For $i\in I$ we can define an algebra automorphism $T_i$ of $U$ by
\begin{align*}
&T_i(K_\mu)=K_{s_i\mu}\qquad(\mu\in \Lambda),\\
&T_i(E_j)=
\begin{cases}
\sum_{k=0}^{-a_{ij}}(-1)^kq_i^{-k}E_i^{(-a_{ij}-k)}E_jE_i^{(k)}\qquad
&(j\in I,\,\,j\ne i),\\
-F_iK_i\qquad
&(j=i),
\end{cases}\\
&T_i(F_j)=
\begin{cases}
\sum_{k=0}^{-a_{ij}}(-1)^kq_i^{k}F_i^{(k)}F_jF_i^{(-a_{ij}-k)}\qquad
&(j\in I,\,\,j\ne i),\\
-K_i^{-1}E_i\qquad
&(j=i).
\end{cases}
\end{align*}
For $w\in W$ we define an algebra automorphism $T_w$ of $U$ by $T_w=T_{i_1}\cdots T_{i_n}$ where $w=s_{i_1}\cdots s_{i_n}$ is a reduced expression.
The automorphism $T_w$ does not depend on the choice of a reduced expression (see Lusztig \cite{Lbook}).

We fix a reduced expression 
\[
w_0=s_{i_1}\cdots s_{i_N}
\]
of $w_0$, and set
\[
\beta_k=s_{i_1}\cdots s_{i_{k-1}}(\alpha_{i_k})\qquad
(1\leqq k\leqq N).
\]
Then we have $\Delta^+=\{\beta_k\mid1\leqq k\leqq N\}$.
For $1\leqq k\leqq N$ set 
\begin{equation}
\label{eq:root vector}
E_{\beta_k}=T_{i_1}\cdots T_{i_{k-1}}(E_{i_k}),\quad
F_{\beta_k}=T_{i_1}\cdots T_{i_{k-1}}(F_{i_k}).
\end{equation}
Then $\{E_{\beta_{N}}^{m_N}\cdots E_{\beta_{1}}^{m_1}\mid
m_1,\dots, m_N\geqq0\}$ (resp. $\{F_{\beta_{N}}^{m_N}\cdots F_{\beta_{1}}^{m_1}\mid
m_1,\dots, m_N\geqq0\}$)
is an $\BF$-basis of $U^+$ (resp. $U^-$), called the PBW-basis (see Lusztig \cite{L2}).
For $1\leqq k\leqq N,\, m\geqq0$ we also set 
\begin{equation}
E^{(m)}_{\beta_k}=E^{m}_{\beta_k}/[m]_{q_{\beta_k}}!,\quad
F^{(m)}_{\beta_k}=F^{m}_{\beta_k}/[m]_{q_{\beta_k}}!,
\end{equation}
where $q_\beta=q^{(\beta,\beta)/2}$ for $\beta\in\Delta^+$.

There exists a bilinear form
\begin{equation}
\label{eq:Drinfeld-paring}
\tau:U^{\geqq0}\times U^{\leqq0}\to\BF,
\end{equation}
called the Drinfeld paring, which is characterized by
\begin{align}
&\tau(x,y_1y_2)=(\tau\otimes\tau)(\Delta(x),y_1\otimes y_2)
&(x\in U^{\geqq0},\,y_1,y_2\in U^{\leqq0}),
\label{eq:DP1}\\
&\tau(x_1x_2,y)=(\tau\otimes\tau)(x_2\otimes x_1,\Delta(y))
&(x_1, x_2\in U^{\geqq0},\,y\in U^{\leqq0}),
\label{eq:DP2}\\
&\tau(K_\lambda,K_\mu)=q^{-(\lambda,\mu)}
&(\lambda,\mu\in \Lambda),
\label{eq:DP3}\\
&\tau(K_\lambda, F_i)=\tau(E_i,K_\lambda)=0
&(\lambda\in \Lambda,\,i\in I),
\label{eq:DP4}\\
&\tau(E_i,F_j)=\delta_{ij}/(q_i^{-1}-q_i)
&(i,j\in I).
\label{eq:DP5}
\end{align}
\begin{proposition}
[\cite{KR}, \cite{KT}, \cite{LS}]
\label{prop:DPPBW}
We have
\begin{align*}
&\tau(E_{\beta_{N}}^{m_N}\cdots E_{\beta_{1}}^{m_1}K_\lambda,
F_{\beta_{N}}^{n_N}\cdots F_{\beta_{1}}^{n_1}K_\mu)\\
=&q^{-(\lambda,\mu)}
\prod_{k=1}^N
\delta_{m_k,n_k}
(-1)^{m_k}
[m_k]_{q_{\beta_k}}!
q_{\beta_k}^{m_k(m_k-1)/2}
(q_{\beta_k}-q_{\beta_k}^{-1})^{-m_k}.
\end{align*}
\end{proposition}
\subsection{Quantized coordinate algebra of $G$}
We denote by $C$ the subspace of $U^*=\Hom_\BF(U,\BF)$ spanned by the matrix coefficients of finite dimensional $U$-modules $E$ such that 
\[
E=\bigoplus_{\lambda\in \Lambda}E_\lambda
\quad\mbox{ with }\quad
E_\lambda=\{v\in E\mid K_\mu v=q^{(\lambda,\mu)}v\,\,(\forall \mu\in \Lambda)\}.
\]
Then $C$ is endowed with a structure of Hopf algebra via
\begin{align*}
&\langle \varphi\psi,u\rangle=
\langle\varphi\otimes\psi,\Delta(u)\rangle
\qquad
&(\varphi,\psi\in C,\,\,u\in U),\\
&\langle 1,u\rangle=
\varepsilon(u)
\qquad
&(u\in U),\\
&\langle \Delta(\varphi),u\otimes u'\rangle=
\langle\varphi,uu'\rangle
\qquad
&(\varphi\in C,\,\,u, u'\in U),\\
&\varepsilon(\varphi)=
\langle\varphi,1\rangle,
\qquad
&(\varphi\in C),\\
&\langle S(\varphi),u\rangle=
\langle\varphi,S(u)\rangle
\qquad
&(\varphi\in C,\,\,u\in U),
\end{align*}
where $\langle\,\,,\,\,\rangle:C\times U\to\BF$ is the canonical paring.
$C$ is also endowed with a structure of $U$-bimodule by
\[
\langle u'\varphi u'',u\rangle=\langle \varphi, u''uu'\rangle
\qquad
(\varphi\in C, u, u', u''\in U).
\]
The Hopf algebra $C$ is a $q$-analogue of the coordinate algebra $\BC[G]$ of $G$ (see \cite{L2}, \cite{T}).

Set 
\[
(U^{\pm})^\bigstar=\bigoplus_{\gamma\in Q^+}\Hom_\BF(U^\pm_{\pm\gamma},\BF)\subset \Hom_\BF(U,\BF).
\]
For $\lambda\in\Lambda$ define an algebra homomorphism $\chi_\lambda:U^0\to\BF$ by $\chi_\lambda(K_\mu)=q^{(\lambda,\mu)}$.
Under the identification $U^-\otimes U^0\otimes U^+\cong U$ of vector spaces we have
\begin{equation}
\label{eq:CF}
C\subset (U^-)^\bigstar\otimes\left(\bigoplus_{\lambda\in\Lambda}\BF\chi_\lambda\right)
\otimes (U^+)^\bigstar
\subset U^*.
\end{equation}

\subsection{Ring of differential operators}
In general for a Hopf algebra $\CH$ over $\BC$ we use the following notation for the comultiplication $\Delta:\CH\to\CH\otimes\CH$:
\[
\Delta(u)=\sum_{(u)}u_{(0)}\otimes u_{(1)}
\qquad(u\in\CH).
\]

We have an $\BF$-algebra structure of
$
D=C\otimes_{\BF} U
$, 
called the Heisenberg double of $C$ and $U$ (see e.g. \cite{Lu}).
It is given by
\[
(\varphi\otimes u)(\varphi'\otimes u')
=\sum_{(u)}\varphi(u_{(0)}\varphi')\otimes u_{(1)}u'\qquad
(\varphi, \varphi'\in C, u, u'\in U)
\]
In our case the algebra $D$ is  an analogue of the ring of differential operators on $G$.
We will identify $U$ and $C$ with subalgebras of $D$ by the embeddings
$U\ni u\mapsto 1\otimes u\in D$ and $C\ni\varphi\mapsto\varphi\otimes1\in D$ respectively.

\subsection{Quantized enveloping algebra of $\Gk$}
The quantized enveloping algebra $V=U_q(\Gk)$ of $\Gk$ is an associative algebra over  $\BF$ with identity element $1$ generated by the elements $Z_\lambda\,(\lambda\in \Lambda),\,X_i, Y_i\,(i\in I)$ satisfying the following defining relations:
\begin{align}
&Z_0=1,\quad 
Z_\lambda Z_\mu=Z_{\lambda+\mu}
&(\lambda,\mu\in \Lambda),
\label{eq:def1V}\\
&Z_\lambda X_iZ_\lambda^{-1}=q^{(\lambda,\alpha_i)}X_i&(\lambda\in \Lambda, i\in I),
\label{eq:def2aV}\\
&Z_\lambda Y_iZ_\lambda^{-1}=q^{(\lambda,\alpha_i)}Y_i
&(\lambda\in \Lambda, i\in I),
\label{eq:def2bV}\\
&X_iY_j-Y_jX_i=0
&(i, j\in I),
\label{eq:def3V}\\
&\sum_{n=0}^{1-a_{ij}}(-1)^nX_i^{(1-a_{ij}-n)}X_jX_i^{(n)}=0
&(i,j\in I,\,i\ne j),
\label{eq:def4V}\\
&\sum_{n=0}^{1-a_{ij}}(-1)^nY_i^{(1-a_{ij}-n)}Y_jY_i^{(n)}=0
&(i,j\in I,\,i\ne j),
\label{eq:def5V}
\end{align}
where 
\[
X_i^{(n)}=
X_i^n/[n]_{q_i}!,
\qquad
Y_i^{(n)}=
Y_i^n/[n]_{q_i}!.
\]

We define subalgebras $V^0, V^{\geqq0}, V^{\leqq0}, V^{+}, V^{-}$ of $V$ by 
\begin{align}
V^0&=\langle Z_\lambda\mid \lambda\in \Lambda\rangle,\\
V^{\geqq0}&=\langle Z_\lambda, X_i\mid \lambda\in \Lambda, i\in I\rangle,\\
V^{\leqq0}&=\langle Z_\lambda, Y_i\mid \lambda\in \Lambda, i\in I\rangle,\\
V^{+}&=\langle X_i\mid i\in I\rangle,\\
V^{-}&=\langle Y_i\mid i\in I\rangle.
\end{align}

Similarly to Proposition \ref{prop:Str-of-U} we have the following.
\begin{proposition}
\label{prop:Str-of-V}
\begin{itemize}
\item[\rm(i)]
$\{Z_\lambda\mid \lambda\in \Lambda\}$ is an $\BF$-basis of $V^0$.
\item[\rm(ii)]
The linear maps 
\begin{gather*}
V^-\otimes V^0\otimes V^+\to V \gets V^+\otimes V^0\otimes V^-,\\
V^+\otimes V^0\to V^{\geqq0} \gets V^0\otimes V^+,\qquad
V^-\otimes V^0\to V^{\leqq0} \gets V^0\otimes V^-
\end{gather*}
induced by the multiplication are all isomorphisms of vector spaces.
\end{itemize}
\end{proposition}
Moreover, we have algebra isomorphisms
\begin{align*}
&\jmath^{\leqq0}:V^{\leqq0}\to U^{\leqq0}
\qquad(Y_i\mapsto F_i,\,Z_\lambda\mapsto K_{-\lambda}),\\
&\jmath^{\geqq0}:V^{\geqq0}\to U^{\geqq0}
\qquad(X_i\mapsto E_i,\,Z_\lambda\mapsto K_{\lambda}).
\end{align*}

We define a bilinear form
\begin{equation}
\sigma:U\times V\to\BF
\end{equation}
by
\begin{multline*}
\sigma(u_+u_0S(u_-),v_-v_+v_0)=
\tau(u_+,\jmath^{\leqq0}(v_-))
\tau(u_0,\jmath^{\leqq0}(v_0))
\tau(\jmath^{\geqq0}(v_+),u_-)\\
\qquad(u_\pm\in U^\pm, u_0\in U^0, v_\pm\in V^\pm, v_0\in V^0).
\end{multline*}
The following result is a consequence of  Gavarini \cite[Theorem 6.2]{Gav}.
\begin{proposition}
\label{prop:invariance}
We have
\[
\sigma(u,vv')=(\sigma\otimes\sigma)(\Delta(u),v\otimes v')
\qquad(u\in U, v, v'\in V).
\]
\end{proposition}

\subsection{$\BA$-forms}
We fix a subring $\BA$ of $\BF$ containing $\BC[q^{\pm1/|\Lambda/Q|}]$.
We denote by $U_{\BA}^L$ the Lusztig ${\BA}$-form of $U$, i.e., $U_{\BA}^L$ is the ${\BA}$-subalgebra of  $U$ generated by the elements
\[
E_i^{(m)},\, F_i^{(m)},\, K_\lambda\qquad
(i\in I,\, m\geqq0,\,\lambda\in \Lambda).
\]
Set 
\begin{align*}
&U_{\BA}^{L,\pm}=U_{\BA}^L\cap U^\pm,\qquad
U_{\BA}^{L,0}=U_{\BA}^L\cap U^0,\\
&U_{\BA}^{L,\geqq0}=U_{\BA}^L\cap U^{L,\geqq0},\qquad
U_{\BA}^{L,\leqq0}=U_{\BA}^L\cap U^{L,\leqq0}.
\end{align*}
Then $U_{\BA}^{L}, U_{\BA}^{L,0}, U_{\BA}^{L,\geqq0}, U_{\BA}^{L,\leqq0}$ are endowed with structures of Hopf algebras over ${\BA}$ via the Hopf algebra structure of $U$, and the multiplication of $U_{\BA}^L$ induces isomorphisms
\begin{align*}
&U_{\BA}^{L}\simeq
U_{\BA}^{L,-}\otimes
U_{\BA}^{L,0}\otimes
U_{\BA}^{L,+}
\simeq
U_{\BA}^{L,+}\otimes
U_{\BA}^{L,0}\otimes
U_{\BA}^{L,-},\\
&U_{\BA}^{L,\geqq0}\simeq
U_{\BA}^{L,0}\otimes
U_{\BA}^{L,+}
\simeq
U_{\BA}^{L,+}\otimes
U_{\BA}^{L,0},\\
&U_{\BA}^{L,\leqq0}\simeq
U_{\BA}^{L,0}\otimes
U_{\BA}^{L,-}
\simeq
U_{\BA}^{L,-}\otimes
U_{\BA}^{L,0}
\end{align*}
of ${\BA}$-modules.
Fix a subset $\Lambda_0$ of $\Lambda$ such that $\Lambda_0\to\Lambda/2Q$ is bijective.
Then $U_{\BA}^{L,+}$, $U_{\BA}^{L,-}$, $U_{\BA}^{L,0}$ are free ${\BA}$-modules with bases
\begin{align*}
&\{E_{\beta_{N}}^{(m_N)}\cdots E_{\beta_{1}}^{(m_1)}\mid
m_1,\dots, m_N\geqq0\},\\
&\{F_{\beta_{N}}^{(m_N)}\cdots F_{\beta_{1}}^{(m_!)}\mid
m_1,\dots, m_N\geqq0\},\\
&
\left\{
K_\lambda
\prod_{i\in I}
\begin{bmatrix}{K_i}\\{n_i}
\end{bmatrix}
\mid
\lambda\in\Lambda_0, \,
n_i\geqq0
\right\}
\end{align*}
respectively, 
where
\[
\begin{bmatrix}{K_i}\\{m}\end{bmatrix}
=\prod_{s=0}^{m-1}
\frac{q_i^{-s}K_i-q_i^sK_i^{-1}}
{q_i^{s+1}-q_i^{-s-1}}
\qquad(m\geqq0).
\]

We denote by $V_{\BA}$ the ${\BA}$-subalgebra of $V$ generated by the elements
\[
X_i^{(m)},\, Y_i^{(m)},\, Z_{\lambda},\,
 \begin{bmatrix}Z_i\\m\end{bmatrix}\qquad
(i\in I,\, m\geqq0, \lambda\in\Lambda),
\]
where
$Z_i=Z_{\alpha_i}$ for $i\in I$ and
\[
\begin{bmatrix}{Z_i}\\{m}\end{bmatrix}
=\prod_{s=0}^{m-1}
\frac{q_i^{-s}Z_i-q_i^sZ_i^{-1}}
{q_i^{s+1}-q_i^{-s-1}}
\qquad(m\geqq0).
\]
Set 
\begin{align*}
&V_{\BA}^{\pm}=V_{\BA}\cap V^\pm,\qquad
V_{\BA}^{0}=V_{\BA}\cap V^0,\\
&V_{\BA}^{\geqq0}=V_{\BA}\cap V^{\geqq0},
\qquad
V_{\BA}^{\leqq0}=V_{\BA}\cap V^{\leqq0}.
\end{align*}
Then the multiplication of $V_{\BA}$ induces isomorphisms
\begin{align*}
&V_{\BA}\simeq
V_{\BA}^{-}\otimes
V_{\BA}^{0}\otimes
V_{\BA}^{+}
\simeq
V_{\BA}^{+}\otimes
V_{\BA}^{0}\otimes
V_{\BA}^{-},\\
&V_{\BA}^{\geqq0}\simeq
V_{\BA}^{0}\otimes
V_{\BA}^{+}
\simeq
V_{\BA}^{+}\otimes
V_{\BA}^{0},\\
&V_{\BA}^{\leqq0}\simeq
V_{\BA}^{0}\otimes
V_{\BA}^{-}
\simeq
V_{\BA}^{-}\otimes
V_{\BA}^{0}
\end{align*}
of ${\BA}$-modules, and we have
\begin{align*}
&\jmath^{\geqq0}(V_{\BA}^{\geqq0})=U^{L,\geqq0}_{\BA},\qquad
\jmath^{\geqq0}(V_{\BA}^{+})=U^{L,+}_{\BA},\qquad
\jmath^{\geqq0}(V_{\BA}^{0})=U^{L,0}_{\BA},
\\
&\jmath^{\leqq0}(V_{\BA}^{\leqq0})=U^{L,\leqq0}_{\BA},\qquad
\jmath^{\leqq0}(V_{\BA}^{-})=U^{L,-}_{\BA},\qquad
\jmath^{\leqq0}(V_{\BA}^{0})=U^{L,0}_{\BA}.
\end{align*}

Set
\begin{align}
&U_{\BA}=\{u\in U\mid\sigma(u,V_{\BA})\subset {\BA}\},\\
&
U_{\BA}^\pm=U^\pm\cap U_{\BA},\qquad
U_{\BA}^{0}=U^0\cap U_{\BA},\\
&U_{\BA}^{\geqq0}=U^{\geqq0}\cap U_{\BA},\qquad
U_{\BA}^{\leqq0}=U^{\leqq0}\cap U_{\BA}.
\end{align}
Then we have
\begin{align}
\label{eq:UA+}
&U^+_{\BA}=\{x\in U^+\mid \tau(x,U_{\BA}^{L,-})\in {\BA}\},\\
\label{eq:UA-}
&U^-_{\BA}=\{y\in U^-\mid \tau(U_{\BA}^{L,+},y)\in {\BA}\},\\
\label{eq:UA0}
&U^0_{\BA}=\sum_{\lambda\in \Lambda}{\BA} K_\lambda,
\end{align}
and the multiplication of $U$ induces isomorphisms
\begin{align*}
&U_{\BA}\simeq
U_{\BA}^{+}\otimes
U_{\BA}^{0}\otimes
U_{\BA}^{-},\\
&U_{\BA}^{\geqq0}\simeq
U_{\BA}^{0}\otimes
U_{\BA}^{+}
\simeq
U_{\BA}^{+}\otimes
U_{\BA}^{0},\\
&U_{\BA}^{\leqq0}\simeq
U_{\BA}^{0}\otimes
U_{\BA}^{-}
\simeq
U_{\BA}^{-}\otimes
U_{\BA}^{0}
\end{align*}
of ${\BA}$-modules.

For $i\in I$ we set
\begin{equation}
A_{i}=(q_{i}-q_{i}^{-1})E_{i},\qquad
B_{i}=(q_{i}-q_{i}^{-1})F_{i}.
\end{equation}
For $1\leqq k\leqq N$ we also set
\begin{equation}
A_{\beta_k}=(q_{\beta_k}-q_{\beta_k}^{-1})E_{\beta_k},\qquad
B_{\beta_k}=(q_{\beta_k}-q_{\beta_k}^{-1})F_{\beta_k}.
\end{equation}
By Proposition \ref{prop:DPPBW} we have the following.
\begin{lemma}
\label{lem:DP-L}
$\{A_{\beta_{N}}^{m_N}\cdots A_{\beta_{1}}^{m_1}\mid
m_1,\dots, m_N\geqq0\}$ (resp. $\{B_{\beta_{N}}^{m_N}\cdots B_{\beta_{1}}^{m_1}\mid
m_1,\dots, m_N\geqq0\}$)
is an ${\BA}$-basis of $U_{\BA}^+$ (resp. $U_{\BA}^-$).
In particular, we have $U_{\BA}^{\pm}\subset U_{\BA}^{L, \pm}$, and $U_{\BA}\subset U_{\BA}^{L}$.
\end{lemma}
It follows that $U_\BA$ coincides with the $\BA$-form of $U$ considered in De Concini-Procesi \cite{DP}.
In particular, we have the following.
\begin{proposition}
\begin{itemize}
\item[\rm(i)]
$U_{\BA}^{0}$,  $U_{\BA}^+$, $U_{\BA}^-$, $U_{\BA}^{\geqq0}$,  $U_{\BA}^{\leqq0}$, $U_{\BA}$ are ${\BA}$-subalgebras of $U$.
\item[\rm(ii)]
$U_{\BA}^0$,  
$U_{\BA}^{\geqq0}$,  $U_{\BA}^{\leqq0}$, $U_{\BA}$ 
are Hopf algebras over ${\BA}$.
\end{itemize}
\end{proposition}
Let $\iota:U_\BA\to U^L_\BA$ be the inclusion.
We denote by
\begin{equation}
\label{eq:sigma-A}
\sigma_{\BA}:U_{\BA}\times V_{\BA}\to {\BA}.
\end{equation}
the bilinear form induced by $\sigma:U\times V\to\BF$.

We set
\begin{align}
C_{\BA}&=
\{\varphi\in C\mid
\langle\varphi,U^L_{\BA}\rangle\subset\BA\},\\
D_\BA&=C_\BA\otimes_{\BA} U_\BA\subset D.
\end{align}
Then $C_{\BA}$ is a Hopf algebra over $\BA$ as well as a $U^L_{\BA}$-bimodule, and $D_\BA$ is an $\BA$-subalgebra of $D$.
It easily follows that
\begin{equation}
\left(
\bigoplus_{\lambda\in\Lambda}\BF\chi_\lambda
\right)
\cap\Hom_\BA(U_\BA^{L,0},\BA)=
\bigoplus_{\lambda\in\Lambda}\BA\chi_\lambda.
\end{equation}
Hence by \eqref{eq:CF} we have
\begin{equation}
\label{eq:CA}
C_\BA=
\left(
(U_\BA^{L,-})^\bigstar\otimes\left(\bigoplus_{\lambda\in\Lambda}\BA\chi_\lambda\right)
\otimes (U_\BA^{L,+})^\bigstar
\right)
\cap C
\subset \Hom_\BA(U_\BA^L,\BA),
\end{equation}
where $(U_\BA^{L,\pm})^\bigstar=\Hom_\BA(U_\BA^{L,\pm},\BA)\cap(U^{\pm})^\bigstar$.

\subsection{Specialization}
For $z\in\BC^\times$ set
\[
\BA_z=\{{f}/{g}\mid
f,g\in\BC[q^{\pm1/|\Lambda/Q|}], g(z)\ne0\}\subset\BF,
\]
and define an algebra homomorphism
\[
\pi_z:\BA_z\to\BC
\]
by $\pi_z(q^{1/|\Lambda/Q|})=z$.
We set
\begin{align*}
&U^L_z=\BC\otimes_{\BA_z} U^L_{\BA_z}, \qquad
V_z=\BC\otimes_{\BA_z} V_{\BA_z},\qquad
U_z=\BC\otimes_{\BA_z} U_{\BA_z},\\
&C_z=\BC\otimes_{\BA_z} C_{\BA_z}, \qquad
D_z=\BC\otimes_{\BA_z} D_{\BA_z}.
\end{align*}
with respect to $\pi_z$.
Then $U^L_z$, $U_z$, $C_z$ are Hopf algebras over $\BC$, and $V_z$, $D_z$ are $\BC$-algebras.
We denote by
\begin{align*}
&\pi^{U^L}_z:U_{\BA_z}^L\to U_z^L,\qquad
\pi^{V}_z:V_{\BA_z}\to V_z,\qquad
\pi^U_z:U_{\BA_z}\to U_z,\\
&\pi^C_z:C_{\BA_z}\to C_z,\qquad
\pi^D_z:D_{\BA_z}\to D_z
\end{align*}
the natural homomorphisms.
We also define 
$U_z^{L,\pm}$, $U_z^{L,0}$, $U_z^{L,\geqq0}$, $U_z^{L,\leqq0}$, 
$V_z^{\pm}$, $V_z^{0}$, $V_z^{\geqq0}$, $V_z^{\leqq0}$, 
$U_z^{\pm}$, $U_z^{0}$, $U_z^{\geqq0}$, $U_z^{\leqq0}$ similarly.
The bilinear form $\sigma_{\BA_z}:U_{\BA_z}\times V_{\BA_z}\to\BA_z$ induces a bilinear form 
\begin{equation}
\label{eq:sigma-zeta}
\sigma_z:U_z\times V_z\to\BC.
\end{equation}
Set 
\[
J_z=\{v\in V_z\mid\sigma_z(U_z,v)=\{0\}\},\qquad
J_z^0=J_z\cap V^0_z.
\]
\begin{lemma}
\label{lem:J}
$J_z$ is a two-sided ideal of $V_z$, and we have
$J_z=V_z^-V_z^+J_z^0$.
In particular, we have
$J_z\cap V^{\geqq0}=V_z^+J_z^0$, and $J_z\cap V^{\leqq0}=V_z^-J_z^0$.
\end{lemma}
\begin{proof}
By Proposition \ref{prop:invariance} $J_z$ is a two-sided ideal.
Set 
$
V'_z=V_z/V_z^-V_z^+J_z^0$.
Since the multiplication of $V_z$ induces an isomorphism
$V_z\simeq V_z^-\otimes V_z^+\otimes V_z^0$,
we have
\[
V'_z\simeq
(V_z^-\otimes V_z^+\otimes V_z^0)/
(V_z^-\otimes V_z^+\otimes J_z^0)
\simeq
V_z^-\otimes V_z^+\otimes (V_z^0/J_z^0).
\]
Let
$\sigma_z':U_z\times V'_z\to\BC$ be the bilinear form induced by $\sigma_z$.
Then we see easily from the definition of $\sigma$ and Proposition \ref{prop:DPPBW} that $\{v\in V'_z\mid\sigma'_z(U_z,v)=\{0\}\}=\{0\}$.
Hence $J_z=V_z^-V_z^+J_z^0$.
\end{proof}
We define an algebra $\overline{V}_z$ by
$
\overline{V}_z=V_z/J_z,
$
and denote by 
$
\overline{\pi}^V_z:V_{\BA_z}\to \overline{V}_z
$
the canonical homomorphism.
Let
$
\overline{\sigma}_z:U_z\times \overline{V}_z\to\BC
$
be the bilinear form induced by \eqref{eq:sigma-zeta}.
Denote the images of $V_z^0$, $V_z^\pm$, $V_z^{\geqq0}$, $V_z^{\leqq0}$ under $V_z\to\overline{V}_z$ by $\overline{V}_z^0$, $\overline{V}_z^\pm$, $\overline{V}_z^{\geqq0}$, $\overline{V}_z^{\leqq0}$ respectively.
Then the multiplication of $\overline{V}_z$ induces isomorphisms
\begin{gather*}
\overline{V}_z
\simeq
\overline{V}_z^-\otimes \overline{V}_z^+\otimes \overline{V}_z^0,\\
\overline{V}_z^{\geqq0}
\simeq
\overline{V}_z^+\otimes \overline{V}_z^0,\qquad
\overline{V}_z^{\leqq0}
\simeq
\overline{V}_z^-\otimes \overline{V}_z^0.
\end{gather*}

Let $\lambda\in\Lambda$.
By abuse of notation we also denote by $\chi_\lambda:U^{L,0}_z\to\BC$ the algebra homomorphism induced by $\chi_\lambda:U\to\BF$.
We see easily the following
\begin{lemma}
\label{lem:chi-ind}
$\{\chi_\lambda\mid\lambda\in\Lambda\}$ is a linearly independent subset of $(U_z^{L,0})^*$.
\end{lemma}
\begin{lemma}
\label{lem:perfect}
The bilinear form $\overline{\sigma}_z$ is perfect in the sense that
\begin{align}
\label{eq:perfect1}
&u\in U_z, \quad\overline{\sigma}_z(u,\overline{V}_z)=\{0\}
\quad\Longrightarrow \quad u=0,\\
\label{eq:perfect2}
&v\in \overline{V}_z, \quad\overline{\sigma}_z(U_z,v)=\{0\}
\quad\Longrightarrow \quad v=0.
\end{align}
\end{lemma}
\begin{proof}
\eqref{eq:perfect2} is clear from the definition.
We see easily from the definition of $\sigma$ and Proposition \ref{prop:DPPBW} that the proof of \eqref{eq:perfect1} is reduced to showing
\[
u\in U^0_z, \quad{\sigma}_z(u,{V}^0_z)=\{0\}
\quad\Longrightarrow \quad u=0.
\]
This follows from Lemma \ref{lem:chi-ind} in view of 
\[
U_z^0=\bigoplus_{\lambda\in\Lambda}\BC K_\lambda,\qquad
V_z^0
\cong
U^{L,0}_z.
\]
\end{proof}

Set 
\begin{gather*}
I_z^0=\jmath^{\geqq0}(J_z^0)\subset U^{L,0}_z,\\
I_z^{\geqq0}=U_z^{L,+}I_z^0
\subset U^{L,\geqq0}_z,\qquad
I_z^{\leqq0}=U_z^{L,-}I_z^0
\subset U^{L,\leqq0}_z,\\
I_z=U_z^{L,-}U_z^{L,+}I_z^0
\subset U^{L}_z.
\end{gather*}
The we have
\begin{equation}
\label{eq:Iz0}
I_z^0=\{
u\in U_z^{L,0}\mid\chi_\lambda(u)=0\,\,(\lambda\in\Lambda)\}.
\end{equation}
\begin{lemma}
$I_z^0$, 
$I_z^{\geqq0}$, 
$I_z^{\leqq0}$, 
$I_z$
are Hopf ideals of 
$U^{L,0}_z$, 
$U^{L,\geqq0}_z$, 
$U^{L,\leqq0}_z$,
$U^{L}_z$  respectively.
\end{lemma}
\begin{proof}
From \eqref{eq:Iz0} we see easily that $I_z^0$ is a Hopf ideal of $U^{L,0}_z$.
It remains to show 
$I_z^0U_z^{L,\pm}\subset U_z^{L,\pm}I_z^0$.
Using $\jmath^{\geqq0}$, $\jmath^{\leqq0}$ we see that this is equivalent to  $J_z^0V_z^{\pm}\subset V_z^{\pm}J_z^0$.
This follows from Lemma \ref{lem:J}.
\end{proof} 
We define a Hopf algebra $\overline{U}^L_z$ by
$
\overline{U}^L_z=U^L_z/I_z,
$
and denote by 
$\overline{\pi}^{U^L}_z:U^L_{\BA_z}\to \overline{U}^L_z
$
the canonical homomorphism.
Denote the images of $U^{L,0}_z$, $U_z^{L,\pm}$, $U_z^{L,\geqq0}$, $U_z^{L,\leqq0}$ under $U^L_z\to\overline{U}^L_z$ by $\overline{U}^{L,0}_z$, $\overline{U}_z^{L,\pm}$, $\overline{U}_z^{L,\geqq0}$, $\overline{U}_z^{L,\leqq0}$ respectively.
We also denote by
\begin{equation}
\overline{\jmath}_z^{\geqq0}:\overline{V}_z^{\geqq0}\to\overline{U}_z^{L,\geqq0},
\qquad
\overline{\jmath}_z^{\leqq0}:\overline{V}_z^{\leqq0}\to\overline{U}_z^{L,\leqq0}
\end{equation}
the algebra isomorphisms induced by $\jmath^{\geqq0}$ and $\jmath^{\leqq0}$.

By \eqref{eq:CA} and Lemma \ref{lem:chi-ind} we have
\begin{equation}
\label{eq:CC}
C_z\subset
(U_z^{L,-})^\bigstar\otimes\left(\bigoplus_{\lambda\in\Lambda}\BC\chi_\lambda\right)
\otimes (U_z^{L,+})^\bigstar
\subset (U_z^L)^*,
\end{equation}
where $(U_z^{L,\pm})^\bigstar=
\BC\otimes_\BA(U_\BA^{L,\pm})^\bigstar
\subset
\Hom_\BC(U_z^{L,\pm},\BC)$.
Hence the natural paring
$\langle\,,\,\rangle:C_z\times U^L_z\to\BC$
descends to 
\[
\langle\,,\,\rangle:C_z\times \overline{U}^L_z\to\BC,
\]
by which the canonical map $C_z\to (\overline{U}^L_z)^*$ is injective.
Moreover, $C_z$ turns out to be a $\overline{U}^L_z$-bimodule.

\subsection{Specialization to $1$}

For an algebraic groups $S$ over $\BC$ with Lie algebra $\Gs$ we will identify the coordinate algebra $\BC[S]$ of $S$ with a subspace of the dual space $U(\Gs)^*$ of the enveloping algebra $U(\Gs)$ by the canonical Hopf paring 
\[
\langle\,,\,\rangle:\BC[S]\otimes U(\Gs)\to \BC
\]
given by
\[
\langle\varphi,u\rangle=(L_u(\varphi))(1)
\qquad(\varphi\in\BC[S], u\in U(\Gs)).
\]
Here, $U(\Gs)\ni u\mapsto L_u\in\End_\BC(\BC[S])$ is the algebra homomorphism given by
\[
(L_a(\varphi))(g)=\frac{d}{dt}\varphi(g\exp(ta))|_{t=0}
\qquad
(a\in\Gg, g\in S, \varphi\in\BC[S]).
\]

We see easily that $J_1$ is generated by the elements
$\pi^V_1(Z_\lambda)-1\in V_1$ for $\lambda\in \Lambda$.
From this we see easily the following.
\begin{lemma}
\label{lem:barv1}
\begin{itemize}
\item[{\rm(i)}]
We have an isomorphism $\overline{V}_1\cong U(\Gk)$ of algebras satisfying
\begin{align*}
&\overline{\pi}^{V}_1({X}_i)\leftrightarrow x_i,\qquad
\overline{\pi}^{V}_1({Y}_i)\leftrightarrow y_i,\\
&\overline{\pi}^{V}_1\left({\begin{bmatrix}Z_i\\m\end{bmatrix}}\right)\leftrightarrow
\begin{pmatrix}t_i\\m\end{pmatrix}:=t_i(t_i-1)\cdots(t_i-m+1)/m!.
\end{align*}
\item[\rm(ii)]
We have an isomorphism $\overline{U}^L_1\cong U(\Gg)$ of Hopf algebras satisfying
\begin{align*}
&\overline{\pi}^{U^L}_1({E}_i)\leftrightarrow e_i,\qquad
\overline{\pi}^{U^L}_1({F}_i)\leftrightarrow f_i,\\
&\overline{\pi}^{U^L}_1\left({\begin{bmatrix}K_i\\m\end{bmatrix}}\right)\leftrightarrow
\begin{pmatrix}h_i\\m\end{pmatrix}:=h_i(h_i-1)\cdots(h_i-m+1)/m!.
\end{align*}
\end{itemize}
\end{lemma}
In the rest of this paper we will occasionally identify $\overline{V}_1$ and $\overline{U}^L_1$ with $U(\Gk)$ and $U(\Gg)$ respectively.

From the identification  $\overline{U}^L_1=U(\Gg)$ we have the following.

\begin{lemma}
The canonical paring 
\[
\langle\,,\,\rangle:C_1\times\overline{U}^L_1\to\BC
\]
induces an isomorphism
\begin{equation}
\label{eq:ISOM-F}
C_1\cong\BC[G]\,(\,\subset U(\Gg)^*\cong(\overline{U}^L_1)^*)
\end{equation} 
of Hopf algebras.
\end{lemma}

In \cite{DP} De Concini-Procesi proved an isomorphism
\begin{equation}
\label{eq:DP}
U_1\cong \BC[K]
\end{equation}
of Poisson Hopf algebras.
They established \eqref{eq:DP} by giving a correspondence between generators of both sides and proving the compatibility after a lengthy calculation.
Later Gavarini \cite{Gav} gave a more natural approach to the isomorphism \eqref{eq:DP}  using the Drinfeld paring.
Namely we have the following.
\begin{proposition}[Gavarini \cite{Gav}]
\label{prop:isom}
The bilinear form 
$
\overline{\sigma}_1:U_1\times\overline{V}_1\to\BC
$
induces a Hopf
 algebra isomorphism
\begin{equation}
\label{eq:ISOM}
\Upsilon:U_1\to\BC[K]\,(\,\subset U(\Gk)^*\simeq\overline{V}_1^*).
\end{equation}
\end{proposition}

The enveloping algebra $U(\Gk^\pm)$ has the direct sum decomposition
\[
U(\Gk^\pm)=\bigoplus_{\beta\in Q^+}U(\Gk^\pm)_{\pm\beta},
\]
where
\[
U(\Gk^\pm)_{\pm\beta},
=\{x\in U(\Gk^\pm)\mid [(h,-h),x]=\beta(h)x\,\,(h\in\Gh)\}
\]
for $\beta\in Q^+$
(note that we have an isomorphism $\Gh\ni h\leftrightarrow(h,-h)\in\Gk^0$).
Then we have 
\[
\BC[K^\pm]=\bigoplus_{\beta\in Q^+}(U(\Gk^\pm)_{\pm\beta})^*
\subset U(\Gk^\pm)^*.
\]
Moreover, 
we have 
\[
\BC[K^0]=\bigoplus_{\lambda\in \Lambda}\BC\hat{\chi}_\lambda\subset U(\Gk_0)^*,
\]
where $\hat{\chi}_\lambda:U(\Gk^0)\to\BC$ is the algebra homomorphism given by $\hat{\chi}_\lambda(h,-h)=\lambda(h)\,\,(h\in\Gh)$.
The isomorphism
\[
K^+\times K^-\times K^0\simeq K\qquad
((g_+,g_-,g_0)\leftrightarrow g_+g_-g_0)
\]
of algebraic varieties induced by the product of the group $K$ gives an identification
\begin{equation}
\label{eq:M-docomp}
\BC[K^+]\otimes\BC[K^-]\otimes\BC[K^0]
\simeq\BC[K]
\end{equation}
of vector spaces.
On the other hand the multiplication of the algebra $U(\Gk)$ induces an identification 
\[
U(\Gk^+)\otimes U(\Gk^-)\otimes U(\Gk^0)\simeq U(\Gk).
\]
Then the canonical embedding $\BC[K]\subset U(\Gk)^*$ is given by
\begin{align*}
\BC[K]&\simeq
\BC[K^+]\otimes\BC[K^-]\otimes\BC[K^0]
\subset
U(\Gk^+)^*\otimes U(\Gk^-)^*\otimes U(\Gk^0)^*\\
&\subset
(U(\Gk^+)\otimes U(\Gk^-)\otimes U(\Gk^0))^*
=U(\Gk)^*.
\end{align*}

For $i\in I$ we define $a_i\in \BC[K^-]\subset U(\Gk^-)^*$,  $b_i\in \BC[K^+]\subset U(\Gk^+)^*$ by
\begin{align*}
&\langle a_i, U(\Gk^-)_{-\beta}\rangle=0\quad
(\beta\ne\alpha_i),\qquad
\langle a_i,y_i\rangle=-1,\\
&\langle b_i, U(\Gk^+)_{\beta}\rangle=0\quad
(\beta\ne\alpha_i),\qquad
\langle b_i,x_i\rangle=1.
\end{align*}
We identify $\BC[K^\pm], \BC[K^0]$ with subalgebras of $\BC[K]$ via \eqref{eq:M-docomp}, and regard $a_i, b_i, \hat{\chi}_\lambda \,\,(i\in I, \lambda\in \Lambda)$ as elements of $\BC[K]$.
By the above argument we see easily the following.
\begin{lemma}
Under the identification \eqref{eq:ISOM} we have
\[
\pi_1^U(A_{i})\leftrightarrow a_i,\quad
\pi_1^U(B_{i})\leftrightarrow b_i\hat{\chi}_{-\alpha_i},\quad
\pi_1^U(K_{\lambda})\leftrightarrow \hat{\chi}_{\lambda}\qquad
(i\in I, \lambda\in \Lambda).
\]
\end{lemma}

Let
$
\iota_1:U_1\to\overline{U}^L_1
$
be the homomorphism induced by the inclusion $\iota:U_{\BA_1}\to U^L_{\BA_1}$.
By Lemma \ref{lem:DP-L} we see easily the following.
\begin{lemma}
\label{lem:iota1}
For $x\in U_1$ we have $\iota_1(x)=\varepsilon(x)1$.
\end{lemma}
From this we obtain the following easily.
\begin{lemma}
\label{lem:com-D}
$D_1$ is a commutative algebra.
In particular, it is identified as an algebra with the coordinate algebra $\BC[G]\otimes\BC[K]$ of $G\times K$.
\end{lemma}

\subsection{Specialization to roots of $1$}
\label{section:roots}
From now on,
we fix an integer $\ell>1$ satisfying
\begin{itemize}
\item[(a)]
$\ell$ is odd,
\item[(b)]
$\ell$ is prime to 3 if $\Gg$ is of type $G_2$,
\item[(c)]
$\ell$ is prime to $|\Lambda/Q|$,
\end{itemize}
and a primitive $\ell$-th root $\zeta\in\BC$ of $1$.
Note that $\pi_\zeta:\BA_\zeta\to\BC$ sends $q$  to $\zeta^{|\Lambda/Q|}$, which is also a primitive $\ell$-th root of $1$ by our assumption (c).

\begin{remark}
{\rm
Denote by $U^{DK}_{\BC[q^{\pm1/|\Lambda/Q|}]}$ the De Concini-Kac 
$\BC[q^{\pm1/|\Lambda/Q|}]$-form of $U$
(see \cite{DK}). 
Namely $U^{DK}_{\BC[q^{\pm1/|\Lambda/Q|}]}$ is the ${\BC[q^{\pm1/|\Lambda/Q|}]}$-subalgebra of $U$ generated by
$\{K_\lambda, E_i, F_i\mid \lambda\in\Lambda, i\in I\}$.
Then we have $U_\zeta\simeq\BC\otimes_{\BC[q^{\pm1/|\Lambda/Q|}]} U^{DK}_{\BC[q^{\pm1/|\Lambda/Q|}]}$
with respect to $q^{1/|\Lambda/Q|}\mapsto\zeta$.}
\end{remark}
We denote by 
$\tilde{\xi}:U^L_\zeta\to U^L_1$
Lusztig's Frobenius morphism (see \cite{L2}).
Namely, $\tilde{\xi}$ is an algebra homomorphism given by
\begin{align}
\label{eq:TXL1}
\tilde{\xi}(
{\pi}^{U^L}_\zeta({E_i^{(n)}}))
&=
\begin{cases}
{\pi}^{U^L}_1({E_i^{(n/\ell)}})\quad&(\ell\,|\,n)\\
0&(\ell\not|\,n),
\end{cases}\\
\label{eq:TXL2}
\tilde{\xi}({\pi}^{U^L}_\zeta({F_i^{(n)}}))
&=
\begin{cases}
{\pi}^{U^L}_1({F_i^{(n/\ell)}})\quad&(\ell\,|\,n)\\
0&(\ell\not|\,n),
\end{cases}
\\
\label{eq:TXL3}
\tilde{\xi}\left({\pi}^{U^L}_\zeta
\left(
{\begin{bmatrix}{K_i}\\{m}\end{bmatrix}}
\right)\right)
&=
\begin{cases}
{\pi}^{U^L}_1
\left(
{\begin{bmatrix}{K_i}\\{m/\ell}\end{bmatrix}}
\right)
\quad&(\ell\,|\,m)\\
\,\,\,0&(\ell\not|\,m),
\end{cases}
\\
\label{eq:TXL4}
\tilde{\xi}({\pi}^{U^L}_\zeta({K}_\lambda))
&=
{\pi}^{U^L}_1({K}_\lambda)\quad(\lambda\in \Lambda).
\end{align}
It is a Hopf algebra homomorphism.
Moreover, for any $\beta\in\Delta^+$ we have
\begin{align}
\label{eq:TXL1a}
\tilde{\xi}(
{\pi}^{U^L}_\zeta({E_\beta^{(n)}}))
&=
\begin{cases}
{\pi}^{U^L}_1({E_\beta^{(n/\ell)}})\quad&(\ell\,|\,n)\\
0&(\ell\not|\,n),
\end{cases}\\
\label{eq:TXL2a}
\tilde{\xi}({\pi}^{U^L}_\zeta({F_\beta^{(n)}}))
&=
\begin{cases}
{\pi}^{U^L}_1({F_\beta^{(n/\ell)}})\quad&(\ell\,|\,n)\\
0&(\ell\not|\,n).
\end{cases}
\end{align}
\begin{lemma}
\label{lem:TXL}
We have $\tilde{\xi}(I_\zeta)\subset I_1$.
\end{lemma}
\begin{proof}
It is sufficient to show
$\tilde{\xi}(I_\zeta^0)\subset I_1^0$.
For $z\in\BC^\times$, $m=(m_i)_{i\in I}\in\BZ_{\geqq0}^I$, and $\nu\in\Lambda_0$
set
\[
K_{m,\nu}(z)=
\pi_z^{U^L}
\left(
K_\nu
\prod_{i\in I}
\begin{bmatrix}{K_i}\\{m_i}\end{bmatrix}
\right)
\in U_z^{L,0}.
\]
Any element $u$ of $U_z^{L,0}$ is uniquely written as a finite sum
\[
u=\sum_{m,\nu}c_{m,\nu}K_{m,\nu}(z)
\qquad(c_{m,\nu}\in\BC).
\]
Then we have $u\in I_z^0$ if and only if
\[
\left.\sum_{m,\nu}c_{m,\nu}q^{(\lambda,\nu)}
\begin{bmatrix}{(\lambda,\alpha_i^\vee)}\\{m_i}\end{bmatrix}_{q_i}
\right|_{q^{1/|\Lambda/Q|}=z}=0
\qquad(\forall\lambda\in \Lambda).
\]
Hence it is sufficient to show that
\begin{equation}
\label{eq:J1}
\left.\sum_{m,\nu}c_{m,\nu}q^{(\lambda,\nu)}
\begin{bmatrix}{(\lambda,\alpha_i^\vee)}\\{m_i}\end{bmatrix}_{q_i}
\right|_{q^{1/|\Lambda/Q|}=\zeta}=0
\qquad(\forall\lambda\in \Lambda)
\end{equation}
implies
\begin{equation}
\label{eq:J2}
\sum_{m,\nu}c_{\ell m,\nu}
\begin{pmatrix}{(\mu,\alpha_i^\vee)}\\{m_i}\end{pmatrix}=0
\qquad(\forall\mu\in \Lambda).
\end{equation}
Indeed \eqref{eq:J2} follows by setting $\lambda=\ell\mu$ in \eqref{eq:J1}.
\end{proof}
We denote by
\begin{equation}
{\xi}:\overline{U}^L_\zeta\to \overline{U}^L_1\,(=U(\Gg))
\end{equation}
the Hopf algebra homomorphism
induced by $\tilde{\xi}$.
By Lusztig \cite{L2} we  have the following.
\begin{proposition}
\label{prop:CGtoF}
There exists a unique linear map 
\begin{equation}
{}^t\xi:C_1\,(=\BC[G])\to C_\zeta
\end{equation}
satisfying
\begin{equation}
\label{eq:tx}
\langle{}^t\xi(\varphi),v\rangle=
\langle \varphi,\xi(v)\rangle\qquad
(\varphi\in C_1,\,\, v\in \overline{U}^L_\zeta).
\end{equation}
It is an injective Hopf algebra homomorphism whose image is contained in the center of $C_\zeta$.
\end{proposition}

\begin{lemma}
There exists an algebra homomorphism
\begin{equation}
\eta:\overline{V}_\zeta\to\overline{V}_1
\end{equation}
such that 
\begin{align*}
\eta(v)
&=(\overline{\jmath}_1^{\geqq0})^{-1}(\xi(\overline{\jmath}_\zeta^{\geqq0}(v)))
&(v\in \overline{V}_\zeta^{\geqq0}),\\
\eta(v)&=(\overline{\jmath}_1^{\leqq0})^{-1}(\xi(\overline{\jmath}_\zeta^{\leqq0}(v)))
&(v\in \overline{V}_\zeta^{\leqq0}).
\end{align*}
\end{lemma}
\begin{proof}
It is sufficient to show that the linear map $\eta:\overline{V}_\zeta\to\overline{V}_1$ defined by
\[
\eta(v_-v_{\geqq0})
=(\overline{\jmath}_1^{\leqq0})^{-1}(\xi(\overline{\jmath}_\zeta^{\leqq0}(v_-)))
(\overline{\jmath}_1^{\geqq0})^{-1}(\xi(\overline{\jmath}_\zeta^{\geqq0}(v_{\geqq0})))
\]
for $v_-\in \overline{V}^-_\zeta, v_{\geqq0}\in\overline{V}^{\geqq0}_\zeta$
is an algebra homomorphism.
This follows easily from 
$[\overline{V}^+_\zeta, \overline{V}^-_\zeta]=0$.
\end{proof}
 By Gavarini \cite[Theorem 7.9]{Gav} we have the following.
\begin{proposition}
\label{prop:UU}
There exists a unique linear map 
\begin{equation}
{}^t\eta:U_1\to U_\zeta
\end{equation}
satisfying
\begin{equation}
\label{eq:te}
\overline{\sigma}_\zeta({}^t\eta(u),v)=
\overline{\sigma}_1(u,\eta(v))\qquad
(u\in U_1,\,\, v\in \overline{V}_\zeta).
\end{equation}
It is an injective Hopf algebra homomorphism whose image is contained in the center of $U_\zeta$.
Moreover, for any $\beta\in\Delta^+$ we have
\[
{}^t\eta(\pi_1^U(A_\beta))=\pi_\zeta^U(A_\beta^\ell),\qquad
{}^t\eta(\pi_1^U(B_\beta))=\pi_\zeta^U(B_\beta^\ell).
\]
\end{proposition}
Let
$
\iota_\zeta:U_\zeta\to\overline{U}^L_\zeta
$
be the homomorphisms induced by $\iota:U_\zeta\to \overline{U}^L_\zeta$.
We see easily the following.
\begin{lemma}
\label{lem:iota}
\begin{itemize}
\item[\rm(i)]
For $x\in U_\zeta$ we have $\xi(\iota_\zeta(x))=\varepsilon(x)1$.
\item[\rm(ii)]
For $y\in U_1$ we have
$\iota_\zeta({}^t\eta(y))=\varepsilon(y)1$.
\end{itemize}
\end{lemma}
\begin{proposition}
\label{prop:center}
The image of the linear map
\[
{}^t\xi\otimes{}^t\eta:
D_1(=C_1\otimes U_1)\to D_\zeta(=C_\zeta\otimes U_\zeta)
\]
is contained in the center of $D_\zeta$.
In particular, ${}^t\xi\otimes{}^t\eta$ 
is an algebra homomorphism.
\end{proposition}
\begin{proof}
Let $\varphi\in C_1$ and $x\in U_\zeta$. 
For $u\in\overline{U}^L_\zeta$ we have
\begin{align*}
&\sum_{(x)}
\langle \iota_\zeta(x_{(0)})\cdot{}^t\xi(\varphi),u\rangle
x_{(1)}=
\sum_{(x)}
\langle {}^t\xi(\varphi),u\iota_\zeta(x_{(0)})\rangle
x_{(1)}\\
=&\sum_{(x)}
\langle \varphi,\xi(u\iota_\zeta(x_{(0)}))\rangle x_{(1)}
=\langle \varphi,\xi(u)\rangle x
=\langle
{}^t\xi(\varphi),u
\rangle
x,
\end{align*}
and hence $x{}^t\xi(\varphi)={}^t\xi(\varphi)x$ in $D_\zeta$.
It follows that ${}^t\xi(\varphi)$ is contained in the center for any
$\varphi\in C_1$.

Let $y\in U_1$.
For $\psi\in C_\zeta$ we have
\[
({}^t\eta(y))\psi
=\sum_{(y)}
(\iota_\zeta({}^t\eta(y_{(0)}))\cdot \psi){}^t\eta(y_{(1)})
=\sum_{(y)}
\varepsilon(y_{(0)})\psi{}^t\eta(y_{(1)})
=\psi({}^t\eta(y)),
\]
and hence ${}^t\eta(y)$ is contained in the center for any
$y\in U_1$.
\end{proof}

\section{Poisson structure arising from quantized enveloping algebras}
The following result is well-known (see \cite{DP}).
\begin{proposition}
\label{prop:Poisson-alg}
Let $\BB$ be a commutative algebra over $\BC$.
We assume that we are given $\hbar\in\BB$ such that $\BB/\hbar\BB\cong\BC$.

Let $\CR$ be a (not necessarily commutative) $\BB$-algebra such that $\hbar:\CR\to \CR$ is injective.
Then the center $Z(\CR/\hbar \CR)$ of $\CR/\hbar \CR$ is endowed with a structure of Poisson algebra by
\[
\{\overline{b}_1,\overline{b}_2\}
=\overline
{\left(\frac{b_1b_2-b_2b_1}\hbar\right)}
\qquad(b_1, b_2\in \CR, \overline{b}_1, \overline{b}_2\in Z(\CR/\hbar \CR)).
\]

Assume moreover that $\CR$ is a Hopf algebra and that there exists a Hopf subalgebra $H$ of $\CR/\hbar \CR$ such that $H\subset Z(\CR/\hbar \CR)$ and $\{H,H\}\subset H$.
Then $H$ is naturally a Poisson Hopf algebra.
\end{proposition}

We will apply this fact to the situation $\BB=\BA_\zeta$, $\hbar=\ell(q^\ell-q^{-\ell})$, and $\CR=C_{\BA_\zeta}, U_{\BA_\zeta}, D_{\BA_\zeta}$.
Note that we have $\BA_\zeta/\ell(q^\ell-q^{-\ell})\BA_\zeta\cong\BC$ by
\[
\Ker\pi_\zeta=\BA_\zeta(q^{1/|\Lambda/Q|}-\zeta)
=\BA_\zeta\ell(q^\ell-q^{-\ell}).
\]
The cases $\CR=C_{\BA_\zeta}, U_{\BA_\zeta}$ is already known.
Namely, we have the following.
\begin{theorem}
[\cite{DL}]
\label{th:DL}
The Hopf subalgebra $\Image{}^t\xi$ of $Z(C_\zeta)$ is closed under the Poisson bracket given in Proposition \ref{prop:Poisson-alg}.
Moreover, the isomorphism $\Image{}^t\xi\cong \BC[G]$ is that of Poisson Hopf algebras, where the Poisson Hopf algebra structure of $\BC[G]$ is the one for $\BC[\Delta G]\cong\BC[G]$ attached to the Manin triple $(\Gg\oplus\Gg,\Delta\Gg,\Gk)$.
\end{theorem}
\begin{theorem}
[\cite{DP}, \cite{Gav}]
\label{th:DPG}
The Hopf subalgebra $\Image{}^t\eta$ of $Z(U_\zeta)$ is closed under the Poisson bracket given in Proposition \ref{prop:Poisson-alg}.
Moreover, the isomorphism $\Image{}^t\eta\cong \BC[K]$ is that of Poisson Hopf algebras, where the Poisson Hopf algebra structure of $\BC[K]$ is the one attached to the Manin triple $(\Gg\oplus\Gg,\Gk,\Delta\Gg)$.
\end{theorem}
In the rest of this paper we will deal with the case where $\CR=D_{\BA_\zeta}$.
The following is the main result of this paper.
\begin{theorem}
\label{thm:main}
The subalgebra $\Image({}^t\xi\otimes{}^t\eta)$ of $Z(D_\zeta)$ is closed under the Poisson bracket  given in Proposition \ref{prop:Poisson-alg}.
Moreover, under the identification
\[
\Image({}^t\xi\otimes{}^t\eta)\cong\BC[G]\otimes \BC[K]
\cong \BC[\Delta G]\otimes \BC[K]
\]
this Poisson algebra structure coincides with the one attached to the Manin triple $(\Gg\oplus\Gg,\Delta\Gg,\Gk)$ as in Proposition \ref{prop:PoissonML}.
\end{theorem}
Set 
\begin{align*}
\CJ&=\Ker(\xi\circ \overline{\pi}^{U^L}_\zeta)\subset U_{\BA_\zeta}^L,\\
\CI&=\{x\in U_{\BA_\zeta}\mid\langle x,V_{\BA_\zeta}\rangle\subset\ell(q^\ell-q^{-\ell}){\BA_\zeta}\}.
\end{align*}
\begin{lemma}
\label{lemma:key}
Let $h\in\Image({}^t\xi)$ and $\varphi\in\Image({}^t\eta)$.
Take $p\in\BC[G]$ and $\Phi\in U_{\BA_\zeta}$ such that $h={}^t\xi(p)$
and $\varphi=\pi_\zeta^U({\Phi})$ respectively.
Assume
\[
\Phi\otimes1-\sum_{(\Phi)}\Phi_{(1)}\otimes\iota(\Phi_{(0)})
\in\ell(q^{\ell}-q^{-\ell})
\sum_r\Psi_r\otimes X_r+
\CI\otimes\CJ\subset U_{\BA_\zeta}\otimes U^L_{\BA_\zeta}
\]
with 
$\Psi_r\in U_{\BA_\zeta}, X_r\in U^L_{\BA_\zeta}$.
Then we have
\[
\{h,\varphi\}=
\sum_r
{}^t\xi(
(\xi\circ \overline{\pi}^{U^L}_\zeta)({X_r})\cdot p)\otimes
\pi_\zeta^U({\Psi_r})
\]
with respect to the Poisson structure of $Z(D_\zeta)$ given in  Proposition \ref{prop:Poisson-alg}.
\end{lemma}
\begin{proof}
Take
$H\in C_{\BA_\zeta}$ such that
$h=\pi_\zeta^C({H})$.
For 
$u\in{U}^L_{\BA_\zeta}, v\in{V}_{\BA_\zeta}$ we see easily that
\begin{align*}
&\langle\{h,\varphi\},\overline{\pi}^{U^L}_\zeta(u)\otimes \overline{\pi}^{V}_\zeta(v)\rangle
\\
=&
\pi_\zeta
\left(
\langle 
H,u
(
\langle \Phi,v\rangle
1
-\sum_{(\Phi)}
\langle
\Phi_{(1)},v\rangle
\iota(\Phi_{(0)})
)
\rangle
/\ell(q^{\ell}-q^{-\ell})
\right)
.
\end{align*}
Write 
\begin{align*}
&\Phi\otimes1-\sum_{(\Phi)}\Phi_{(1)}\otimes\iota(\Phi_{(0)})
=\ell(q^{\ell}-q^{-\ell})
\sum_r\Psi_r\otimes X_r+
\sum_s\Xi_s\otimes Y_s,
\end{align*}
where $\Xi_s\in\CI, Y_s\in \CJ$.
Then we have
\begin{align*}
&\langle\{h,\varphi\},\overline{\pi}^{U^L}_\zeta(u)\otimes \overline{\pi}^{V}_\zeta(v)\rangle
\\
=&
\sum_r
\pi_\zeta(\langle \Psi_r,v\rangle)
\pi_\zeta\left(
\left\langle 
H,uX_r
\right\rangle
\right)
+
\sum_s
\pi_\zeta\left(
\frac{\langle \Xi_s,v\rangle}
{\ell(q^{\ell}-q^{-\ell})}
\right)
\pi_\zeta
\left(
\left\langle 
H,u
Y_s
\right\rangle
\right)
\\
=&
\sum_r
\langle \pi_\zeta^U(\Psi_r),\overline{\pi}^{V}_\zeta(v)\rangle
\left\langle 
h,\overline{\pi}^{U^L}_\zeta(u)\overline{\pi}^{U^L}_\zeta(X_r)
\right\rangle\\
&\qquad
+
\sum_s
\pi_\zeta\left(
\frac{\langle \Xi_s,v\rangle}
{\ell(q^{\ell}-q^{-\ell})}
\right)
\left\langle 
h,\overline{\pi}^{U^L}_\zeta(u)\overline{\pi}^{U^L}_\zeta(Y_s)
\right\rangle.
\end{align*}
By $h={}^t\xi(p)$ we have
\begin{align*}
&\left\langle 
h,\overline{\pi}^{U^L}_\zeta(u)\overline{\pi}^{U^L}_\zeta(X_r)
\right\rangle
=
\left\langle 
p,(\xi\circ \overline{\pi}^{U^L}_\zeta)(u)(\xi\circ \overline{\pi}^{U^L}_\zeta)(X_r)
\right\rangle\\
=&
\left\langle 
(\xi\circ \overline{\pi}^{U^L}_\zeta)(X_r)\cdot p,(\xi\circ \overline{\pi}^{U^L}_\zeta)(u)
\right\rangle
=
\left\langle 
{}^t\xi((\xi\circ \overline{\pi}^{U^L}_\zeta)(X_r)\cdot p),\overline{\pi}^{U^L}_\zeta(u)
\right\rangle.
\end{align*}
Similarly, we have
\[
\left\langle 
h,\overline{\pi}^{U^L}_\zeta(u)\overline{\pi}^{U^L}_\zeta(Y_s)
\right\rangle
=
\left\langle 
p\cdot(\xi\circ \overline{\pi}^{U^L}_\zeta)(u),(\xi\circ \overline{\pi}^{U^L}_\zeta)(Y_s)
\right\rangle=0.
\]
Now the assertion is clear.
\end{proof}

Now let us show Theorem \ref{thm:main}.
By Theorem \ref{th:DL} and Theorem \ref{th:DPG}
it is sufficient to show that for $h\in\Image({}^t\xi), \varphi\in\Image({}^t\eta)$ our Poisson bracket 
$\{h,\varphi\}$ defined above coincides with the one coming from the Manin triple.
In order to avoid confusion 
we denote by $\{\,,\,\}'$ the Poisson bracket of 
$\BC[G]\otimes\BC[K]$ coming from the Manin triple.
We need to show
\begin{equation}
\label{eq:formula2}
\{h,\varphi\}=\{h,\varphi\}'\qquad(\forall h\in\Image({}^t\xi))
\end{equation}
for any $\varphi\in \Image({}^t\eta)$.
If \eqref{eq:formula2} holds for $\varphi\in \Image({}^t\eta)$,
we have
\begin{equation}
\label{eq:formula1}
\{f,\varphi\}=\{f,\varphi\}'\qquad(\forall f\in\Image({}^t\xi\otimes{}^t\eta))
\end{equation}
by
\begin{align*}
\{h\psi,\varphi\}
=\{h,\varphi\}\psi+h\{\psi,\varphi\}
=\{h,\varphi\}'\psi+h\{\psi,\varphi\}'=\{h\psi,\varphi\}'
\end{align*}
for $h\in\Image({}^t\xi), \psi\in\Image({}^t\eta)$.
Hence for each $\varphi\in \Image({}^t\eta)$ \eqref{eq:formula2} is equivalent to \eqref{eq:formula1}.
Then it follows from the definition of the Poisson algebra  that \eqref{eq:formula2} for $\varphi=\varphi_1, \varphi=\varphi_2$ imply
those for $\varphi=\varphi_1\varphi_2$, $\varphi=\{\varphi_1,\varphi_2\}$.
Therefore it is sufficient to show 
\eqref{eq:formula2} in the cases where $\varphi$ belongs to a generator system of the Poisson algebra $\Image({}^t\eta)$.
By \cite{DP} the Poisson algebra $\BC[K]$ is generated by the elements of the form $\hat{\chi}_\lambda, a_i, b_i$ for $\lambda\in\Lambda, i\in I$.
Under the isomorphism $\BC[K]\cong\Image({}^t\eta)$ of Poisson algebras 
we have
\begin{align*}
\hat{\chi}_{\lambda}
&\longleftrightarrow
\pi_\zeta^U({K_{\ell\lambda}})
&(\lambda\in\Lambda),\\
a_i\hat{\chi}_{-\alpha_i}
&\longleftrightarrow
\pi_\zeta^U({(q_i-q_i^{-1})^\ell E_i^{\ell}K_i^{-\ell}})
&(i\in I),\\
b_i\hat{\chi}_{-\alpha_i}
&\longleftrightarrow
\pi_\zeta^U({(q_i-q_i^{-1})^\ell F_i^{\ell}})
&(i\in I).
\end{align*}
Hence we have only to show \eqref{eq:formula2} in the cases
\[
\varphi=\pi_\zeta^U({K_{\ell\lambda}}),\qquad
\varphi=\pi_\zeta^U({(q_i-q_i^{-1})^\ell E_i^\ell K_i^{-\ell}}),\qquad
\varphi=\pi_\zeta^U({(q_i-q_i^{-1})^\ell F_i^\ell})
\]
for $\lambda\in\Lambda, i\in I$.

For bases $\{X_r\}$ and $\{Y_r\}$ of 
$\Gg$ and $\Gk$ respectively such that
$\rho((X_r,X_r),Y_s)=\delta_{rs}$
we have
\[
\{h,\varphi\}'
=\sum_r L_{X_r}(h) R_{Y_r}(\varphi)
\qquad(h\in\BC[G], \varphi\in\BC[K]).
\]
From this we can easily deduce
\begin{align}
&\{h,\hat{\chi}_{\lambda}\}'
=
-\frac12L_{H_\lambda}(h)\hat{\chi}_{\lambda}
&(\lambda\in\Lambda),\\
&\{h,a_i\hat{\chi}_{-\alpha_i}\}'
=-\frac{(\alpha_i,\alpha_i)}2L_{e_i}(h)\hat{\chi}_{-\alpha_i}
&(i\in I)
,\\
&\{h,b_i\hat{\chi}_{-\alpha_i}\}'
=-\frac{(\alpha_i,\alpha_i)}2L_{f_i}(h)\hat{\chi}_{-\alpha_i}
&(i\in I),
\end{align}
where $H_\lambda\in\Gh$ is given by $\kappa(H_\lambda,H)=\lambda(H)\quad(H\in\Gh)$.

Let us show \eqref{eq:formula2} for $\varphi=\pi_\zeta^U({(q_i-q_i^{-1})^\ell F_i^{\ell}})$.
For
$\Phi=(q_i-q_i^{-1})^\ell F_i^{\ell}$ we have
\begin{align*}
&\Phi\otimes1-\sum_{(\Phi)}\Phi_{(1)}\otimes\iota(\Phi_{(0)})\\
=&
(q_i-q_i^{-1})^\ell 
\left(F_i^\ell\otimes1
-\sum_{r=0}^\ell
q_i^{r(\ell-r)}
\begin{bmatrix}
\ell\\
r
\end{bmatrix}_{q_i}
[r]!_{q_i}
F_i^{\ell-r}K_i^{-r}\otimes
F_i^{(r)}
\right)\\
\in&
(q_i-q_i^{-1})^\ell 
\left(F_i^\ell\otimes1
-(F_i^\ell\otimes1+[\ell]!_{q_i}K_i^{-\ell}\otimes F_i^{(\ell)})
\right)+\CI\otimes \CJ\\
=&
\ell(q^\ell-q^{-\ell})
\left(
-\frac{(q_i-q_i^{-1})^\ell 
[\ell]!_{q_i}}{\ell(q^\ell-q^{-\ell})}K_i^{-\ell}\otimes F_i^{(\ell)}
\right)
+\CI\otimes \CJ.
\end{align*}
Hence the assertion follows from 
\[
\left.\frac{(q_i-q_i^{-1})^\ell 
[\ell]!_{q_i}}{\ell(q^\ell-q^{-\ell})}\right|_{q^{1/|\Lambda/Q|}=\zeta}=\frac{(\alpha_i,\alpha_i)}2,
\]
which is easily checked.
The verification of \eqref{eq:formula2} for
$\varphi=\pi_\zeta^U({(q_i-q_i^{-1})^\ell E_i^{\ell}K_i^{-\ell}})$ is similar and omitted.

Let us finally show \eqref{eq:formula2} for
$\varphi=\pi_\zeta^U({K_{\ell\lambda}})$.
We need to show
\[
\{{}^t\xi(p),\varphi\}
=
-\frac12
{}^t\xi(L_{H_\lambda}(p))\otimes\varphi
\]
for $p\in\BC[G]$.
Take $H\in C_{\BA_\zeta}$ such that 
$\pi_\zeta^C({H})={}^t\xi(p)$.
For $z\in\BC^\times$ and $\nu\in\Lambda$ we set
\begin{align*}
(C_{\BA_z})_\nu
&=\{\varphi\in C_{\BA_z}
\mid
u\cdot\varphi=\chi_\nu(u)\varphi\,\,(u\in U_{\BA_z}^{L,0})\},\\
(C_z)_\nu
&=\{\varphi\in C_{z}
\mid
u\cdot\varphi=\chi_\nu(u)\varphi\,\,(u\in U_{z}^{L,0})\}.
\end{align*}
Then we have
\[
{}^t\xi(\BC[G]_\nu)\subset
(C_\zeta)_{\ell\nu}=\pi_\zeta^C((C_{\BA_\zeta})_{\ell\nu})
\qquad(\nu\in\Lambda),
\]
and hence 
we may assume $p\in\BC[G]_\nu$ and $H\in(C_{\BA_\zeta})_{\ell\nu}$.
For
$\Phi=K_{\ell\lambda}$
we have
\[
\Phi\otimes1-\sum_{(\Phi)}\Phi_{(1)}\otimes\iota(\Phi_{(0)})=K_\lambda^{\ell}\otimes 1-K_\lambda^{\ell}\otimes\iota(K_\lambda^{\ell})
=-K_\lambda^{\ell}\otimes(\iota(K_\lambda^{\ell})-1).
\]
Hence for 
$u\in{U}^L_{\BA_\zeta}, v\in{V}_{\BA_\zeta}$ we have
\begin{align*}
&\langle\{{}^t\xi(p),\varphi\},\overline{\pi}_\zeta^{U^L}({u})\otimes \overline{\pi}_\zeta^{V}({v})\rangle
\\
=&\pi_\zeta
\left(
\left\langle 
H,u
\left(
\langle \Phi,v\rangle
1
-\sum_{(\Phi)}
\langle
\Phi_{(1)},v\rangle
\iota(\Phi_{(0)})
\right)
\right\rangle
/\ell(q^{\ell}-q^{-\ell})
\right)
\\
=&
-
\pi_\zeta(
\langle K_{\ell\lambda},v\rangle
\left\langle 
H,u
(\iota(K_{\ell\lambda})-1)
\right\rangle
/\ell(q^{\ell}-q^{-\ell})
)\\
=&
-
\pi_\zeta(
\langle K_{\ell\lambda},v\rangle
\left\langle 
(\iota(K_{\ell\lambda})-1)\cdot H,u
\right\rangle
/\ell(q^{\ell}-q^{-\ell})
)\\
=&
-
\pi_\zeta(
(q^{\ell^2(\lambda,\nu)}-1)
/\ell(q^{\ell}-q^{-\ell})
)
\,
\pi_\zeta(
\langle K_{\ell\lambda},v\rangle
)
\pi_\zeta(
\left\langle 
H,u
\right\rangle
)\\
=&
-\frac{\ell(\lambda,\nu)}{2\ell}
\langle \varphi,\overline{\pi}_\zeta^{V}({v})\rangle
\left\langle 
{}^t\xi(p),\overline{\pi}_\zeta^{U^L}(u)
\right\rangle\\
=&
-\frac12
\langle
{}^t\xi(L_{H_\lambda}(p))\otimes\varphi,\overline{\pi}_\zeta^{U^L}({u})\otimes \overline{\pi}_\zeta^{V}({v})
\rangle
.
\end{align*}
The proof of Theorem \ref{thm:main} is complete.

\bibliographystyle{unsrt}

\end{document}